\documentclass[preprint]{article}

\usepackage{amsmath, amsthm, amsfonts,amssymb}
\usepackage{bm}
\usepackage[left=1in,right=1in,top=1in,bottom=1in]{geometry}

\usepackage{fancyhdr}
\usepackage{bbm}
\usepackage[affil-it]{authblk}
\usepackage{graphicx, color}
\usepackage{subfigure}
\usepackage{epstopdf}
\usepackage{float}
\usepackage{hyperref}
\usepackage{tikz}
\usetikzlibrary{matrix,patterns,decorations.pathreplacing}

\theoremstyle{plain}
\newtheorem{thm}{Theorem}[section]

\newtheorem{lem}[thm]{Lemma}

\newtheorem{defn}[thm]{Definition}
\newtheorem{assm}[thm]{Assumption}
\theoremstyle{remark}

\newtheorem{rem}[thm]{Remark}

\tikzset{
  mynode/.style={fill,circle,inner sep=2pt,outer sep=0pt}
}

\numberwithin{equation}{section}

\renewcommand{\Im}{{\rm{Im}}}
\renewcommand{\Re}{{\rm{Re}}}

\date{}
\begin{document}
\title{A necessary and sufficient condition for edge universality at the largest singular values of covariance matrices}

\author[1]{Xiucai Ding \thanks{E-mail: xiucai.ding@mail.utoronto.ca.}}
\author[2]{Fan Yang  \thanks{E-mail: fyang75@math.wisc.edu.}}
\affil[1]{Department of Statistical Sciences, University of Toronto}
\affil[2]{Department of Mathematics, University of Wisconsin-Madison}

\maketitle
\begin{abstract}
In this paper, we prove a necessary and sufficient condition for the edge universality of sample covariance matrices with general population. We consider sample covariance matrices of the form $\mathcal Q = TX(TX)^{*}$, where the sample $X$ is an $M_2\times N$ random matrix with $i.i.d.$ entries with mean zero and variance $N^{-1}$, and $T$ is an $M_1 \times M_2$ deterministic matrix satisfying $T^* T$ is diagonal. We study the asymptotic behavior of the largest eigenvalues of  $\mathcal Q$  when $M:=\min\{M_1,M_2\}$ and $N$ tends to infinity with $\lim_{N \to \infty} {N}/{M}=d \in (0, \infty)$. Under mild assumptions of $T$, we prove that the Tracy-Widom law holds for the largest eigenvalue of $\mathcal Q$ if and only if $\lim_{s \rightarrow \infty}s^4 \mathbb{P}(\vert \sqrt{N} x_{ij} \vert \geq s)=0$. This condition was first proposed for Wigner matrices by Lee and Yin \cite{LY}.
\end{abstract}

\section{Introduction}

Sample covariance matrices are fundamental objects in modern multivariate statistics. In the classical setting \cite{AT}, for an $M \times N $ sample matrix $X$, people focus on the asymptotic properties of $X X^{*}$ when $M$ is fixed and $N$ goes to infinity. In this case the central limit theorem and law of large number can be applied to the statistical inference procedure. However, the advance of technology has led to high dimensional data such that $M$ is comparable to or even larger than $N$ \cite{IJ,IJ2}. This high dimensionality can not be handled with the classical multivariate statistical theory.

An important topic in the statistical study of sample covariance matrices is the distributions of the largest eigenvalues, which have been playing essential roles in analyzing the data matrices. For example, they are of great interest to the principal component analysis (PCA) \cite{JI}, which is a standard technique for dimensionality reduction and provides a way to identify patterns from real data. 
Also, the largest eigenvalues are commonly used in hypothesis testing, such as the well-known Roy's largest root test \cite{NJ}. 
For a detailed review, one can refer to \cite{IJ,PA,YZB}.

In this paper, we study the largest eigenvalues of sample covariance matrices with comparable dimensions and general population (i.e. the expectation of the sample covariance matrices are non-scalar matrices). More specifically, we consider sample covariance matrices of the form $\mathcal Q = TX(TX)^{*}$, where the sample $X=(x_{ij})$ is an $M_2 \times N$ random matrix with $i.i.d.$ entries such that $\mathbb E x_{11}=0$ and $\mathbb E |x_{11}|^2 = N^{-1}$, and $T$ is an $M_1 \times M_2$ deterministic matrix. On dimensionality, we assume that $N/M \to d$ as $N\to \infty$, where $M:=\min\{M_1, M_2\}$. In the last decade, random matrix theory has been proved to be one of the most powerful tools in dealing with this kind of large dimensional random matrices. It is well-known that the empirical spectral distribution (ESD) of $\mathcal Q$ converges to the (deformed) Marchenko-Pastur (MP) law \cite{MP}, whose rightmost edge $\lambda_r$ gives the asymptotic location of the largest eigenvalue. Furthermore, it was proved in a series of papers that under a proper $N^{2/3}$ scaling, the distribution of the largest eigenvalue $\lambda_1$ of $\mathcal Q$ around $\lambda_r$ converges to the Tracy-Widom distribution \cite{TW1,TW}, which arises as the limiting distribution of the largest rescaled eigenvalues of the Gaussian orthogonal ensemble (GOE). This result is commonly referred to as the {\it{edge universality}}, in the sense that it is independent of the detailed distribution of the entries of $X$. The Tracy-Widom distribution of $(\lambda_1 - \lambda_r)$ was first proved for $\mathcal Q$ with $X$ consisting of $i.i.d.$ centered real or complex Gaussian random entries (i.e. $X$ is a Wishart matrix) and trivial population (i.e. $T=I$) \cite{IJ2}. The edge universality in the $T=I$ case were later proved for all random matrices $X$ whose entries satisfy arbitrary sub-expoenetial distribution \cite{PY2,PY}. When $T$ is a (non-scalar) diagonal matrix, the Tracy-Widom distribution was proved for Wishart matrix $X$ first in \cite{Karoui} (non-singular $T$ case) and \cite{Onatski} (singular $T$ case). Later the edge universality for diagonal $T$ was proved in \cite{BPZ1,LS} for all random matrices $X$ with sub-expoenetial distributed entries. The most general case with rectangular and non-diagonal $T$ is considered in \cite{KY2}, where the edge universality was proved for $X$ with sub-expoenetial distributed entries.

In this paper, 
we prove a necessary and sufficient condition for the edge universality of sample covariance matrices with general population. Briefly speaking, we will prove the following result. 

\vspace{5pt}
{\it{If $T^* T$ is diagonal and satisfies some mild assumptions, then $N^{\frac{2}{3}}\left(\lambda_1(\mathcal Q)-\lambda_r\right)$ converges weakly to the Tracy-Widom distribution if and only if the entries of $X$ satisfy the following tail condition:}}
\begin{equation}
\lim_{s \rightarrow \infty } s^4 \mathbb{P}\left( \sqrt{N}\vert x_{11} \vert \geq s\right)=0 . \label{INTRODUCTIONEQ1}
\end{equation}

For a precise statement of the result, one can refer to the Theorem \ref{main_thm}. Note that under the assumption $T^* T$ is diagonal, the matrix $\mathcal Q$ is equivalent (in terms of eigenvalues) to a sample covariance matrix with diagonal $T$. Hence our result is basically an improvement of the ones in \cite{BPZ1,LS}. The condition (\ref{INTRODUCTIONEQ1}) provides a simple criterion for the edge universality of sample covariance matrices without assuming any other properties of matrix entries. 

Note that the condition (\ref{INTRODUCTIONEQ1}) is slightly weaker than the finite fourth moment (of $\sqrt{N}x_{11}$) condition. In the null case with $T=I$, it was proved before in \cite{YBK} that $\lambda_1 \rightarrow \lambda_r$ almost surely if the fourth moment exists. Later the finite fourth moment condition is proved to be also necessary for the almost sure convergence of $\lambda_1$ in \cite{BSY}. Our theorem, however, shows that the existence of finite fourth moment is not necessary for the Tracy-Widom fluctuation. In fact, one can easily construct random variables that satisfies condition (\ref{INTRODUCTIONEQ1}) but has infinite fourth moment. For example, we can use the following probability density function with $x^{-5}(\log x)^{-1}$ tail:
\begin{equation*}
\rho(x)=\frac{e^4(4 \log x +1)}{x^5(\log x)^2} \mathbf{1}_{\{x>e\}}.
\end{equation*}
Then in this case $\lambda_1$ does not converge to $\lambda_r$ almost surely, but $N^{{2}/{3}}(\lambda_1-\lambda_r)$ still converges weakly to the Tracy-Widom distribution. On the other hand, Silverstein derived that $\lambda_1 \rightarrow \lambda_r$ in probability from the condition (\ref{INTRODUCTIONEQ1}) \cite{SW}. So our result can be also regarded as an improvement of the one in \cite{SW}.


The necessary and sufficient condition for the edge universality of Wigner matrix ensembles has been proved by Lee and Yin in \cite{LY}. The main idea of our proof is similar to theirs. For the necessary part, the key observation is that if the condition (\ref{INTRODUCTIONEQ}) does not hold, then $X$ has a large entry with nonzero probability. As a result, the largest eigenvalue of $\mathcal Q$ can be larger than $C$ with nonzero probability for any fixed constant $C>\lambda_r$, i.e. $\lambda_1 \not\to \lambda_r$ in probability. The sufficient part is more delicate. A key observation of \cite{LY} is that if we introduce a ``cutoff" on the matrix elements of $X$ at the level $N^{-\epsilon}$, then the matrix with cutoff can well approximate the original matrix in terms of the largest singular value if and only if the condition (\ref{INTRODUCTIONEQ}) holds. Thus the problem is reduced to proving the edge universality of sample covariance matrices with size $\le N^{-\epsilon}$. In \cite{BPZ1,LS}, the edge universality for sample covariance matrices have been proved by assuming a subexponential decay of the $x_{ij}$ entries. We first extend their edge universality results to sample covariance matrices with entries having size $\le N^{-\phi}$ for some $1/3<\phi \le 1/2$; see Lemma \ref{lem_small} and Lemma \ref{lem_smallcomp}. Then a major part of this paper is devoted to extending the ``small" support $N^{-\phi}$, $1/3<\phi \le 1/2$, case to the ``large" support $N^{-\epsilon}$ case. This goal can be accomplished with a Green function comparison method, which has been applied successfully in proving the universality of covariance matrices \cite{PY2,PY}. A technical difficulty is that the change of $\mathcal Q$ is nonlinear in terms of the change of the matrix $X$. To handle this, we use the self-adjoint linearization trick; see Definition \ref{linearize_block}.

This paper is organized as follows. In Section \ref{main_result}, we define the deformed Marchenko-Pastur law and its rightmost edge (i.e. the soft edge) $\lambda_r$, and then give the main theorem of this paper. In Section \ref{sec_maintools}, we introduce the notations and collect some tools that will be used to prove the main theorem. In Section \ref{sec_cutoff}, we prove the main result. In Sections \ref{tools} and \ref{comparison}, we prove some key lemmas and theorems that are used in the proof of main result. In particular, the Green function comparison is performed in Section \ref{comparison}. In Appendix \ref{appendix1}, we prove the local law and edge universality of sample covariance matrices with small support $N^{-\phi}$ with $1/3<\phi \le 1/2$.

\begin{rem}
In this paper, we do not consider the edge universality at the leftmost edge (i.e. the hard edge) for the smallest eigenvalues. It will be studied elsewhere. Let $\lambda_l$ be the leftmost edge of the deformed Marchenko-Pastur law. It is worth mentioning that the condition (\ref{INTRODUCTIONEQ1}) can be shown to be sufficient for the edge universality at the hard edge if $\lambda_r \not\to 0$ as $N\to \infty$. However, it seems that (\ref{INTRODUCTIONEQ1}) is not necessary. So far, there is no conjecture about the necessary and sufficient condition for the edge universality at the  hard edge.
\end{rem}

\vspace{2pt}

\noindent{\bf Conventions.} All quantities that are not explicitly constant may depend on $N$, and we usually omit $N$ from our notations. We use $C$ to denote a generic large positive constant, whose value may change from one line to the next. Similarly, we use $\epsilon$, $\tau$ and $c$ to denote generic small positive constants. For two quantities $a_N$ and $b_N$ depending on $N$, the notation $a_N = O(b_N)$ means that $|a_N| \le C|b_N|$ for some positive constant $C>0$, and $A_N=o(B_N)$ means that $|a_N| \le c_N |b_N|$ for some positive constants $c_N\to 0$ as $N\to \infty$. We also use the notation $a_N \sim b_N$ if $a_N = O(b_N)$ and $b_N = O(a_N)$. For a matrix $A$, we use $\|A\|$ to denote its operator norm and $\|A\|_{HS}$ the Hilbert-Schmidt norm; for a vector $\mathbf v=(v_i)_{i=1}^n$, $\|\mathbf v\|\equiv \|\mathbf v\|_2$ stands for the Euclidean norm, while $|\mathbf v| \equiv \|\mathbf v\|_1$ stands for the $l^1$-norm. In this paper, we usually write an $n\times n$ identity matrix $I_{n\times n}$ as $1$ or $I$ when there is no confusion about the dimension. If two random variables $X$ and $Y$ have the same distribution, we write $X\stackrel{d}{=} Y$.

\section{Definitions and Main Result}\label{main_result}

\subsection{Sample covariance matrices with general populations}

We consider the $M_1 \times M_1$ sample covariance matrix $\mathcal Q_1:=TX(TX)^*$, where $T$ is a deterministic $M_1\times M_2$ matrix and $X$ is a random $M_2\times N$ matrix. We assume $X=(x_{ij})$ have independent entries $x_{ij}= N^{-1/2}q_{ij}$, $1 \leq i \leq M_2$ and $1 \leq j \leq N$, where $q_{ij}$ are {\it{i.i.d.}} random variables satisfying
\begin{equation}\label{assm1}
\mathbb{E} q_{11} =0, \ \quad \ \mathbb{E} \vert q_{11} \vert^2  =1.
\end{equation}
In this paper, we regard $N$ as the fundamental parameter and $M_{1,2}\equiv M_{1,2}(N)$ as depending on $N$. We define $M:=\min\{M_1,M_2\}$ and the aspect ratio $d_N:= N/M$. Moreover, we assume that
\begin{equation}
 d_N \to d \in (0,\infty), \ \ \text{ as } N\to \infty. \label{assm2}
 \end{equation}
For simplicity, we assume that $N/M$ is constant and hence use $d$ instead of $d_N$. We denote the eigenvalues of $\mathcal Q_1$ in decreasing order as $\lambda_1(\mathcal Q_1)\geq \ldots \geq \lambda_{M_1}(\mathcal Q_1)$. We will also use the $N \times N$ matrix $\mathcal Q_2:=(TX)^* TX$ and its eigenvalues $\lambda_1(\mathcal Q_2) \geq \ldots \geq \lambda_N(\mathcal Q_2)$. Since $\mathcal Q_1$ and $\mathcal Q_2$ share the same nonzero eigenvalues, we will for simplicity write $\lambda_j$, $1\le j \le \min\{N,M_1\}$, to denote the $j$-th eigenvalue of both $\mathcal Q_1$ and $\mathcal Q_2$ without causing any confusion. 

We assume that $T^* T$ is diagonal. In other words, $T$ has a singular decomposition $T=U\bar D$, where $U$ is an $M_1 \times M_1$ unitary matrix and $\bar D=(D,0)$ is an $M_1 \times M_2$ matrix with diagonal blocks. Then it is equivalent to study the eigenvalues of $\bar DX(\bar DX)^*$. When $M_1 \le M_2$ (i.e. $M=M_1$), we can write $\bar D=(D,0)$ where $D$ is an $M\times M$ diagonal matrix such that $D_{11} \ge \ldots \ge D_{MM}.$ Hence we have $\bar D X = D\tilde X$, where $\tilde X$ is the upper $M\times N$ block of $X$ with {\it{i.i.d.}} entries $x_{ij}$, $1 \leq i \leq M$ and $1 \leq j \leq N$. On the other hand, when $M_1 \ge M_2$ (i.e. $M=M_2$), we can write $\bar D=\begin{pmatrix}D\\ 0\end{pmatrix}$ with $D$ being an $M\times M$ diagonal matrix as above. Then $\bar D X = \begin{pmatrix}D X \\ 0\end{pmatrix}$, which shares the same nonzero singular values with $DX$. The above discussions show that we can make the following stronger assumption on $T$:
\begin{equation}\label{simple_assumption}
M_1 = M_2 = M, \ \ \text{ and } \ \ T \equiv D = \text{diag}\left(\sigma_1^{1/2},\sigma_2^{1/2},\ldots,\sigma_N^{1/2}\right) \text{ with } \sigma_1 \ge \sigma_2 \ge \ldots \ge \sigma_M \ge 0.
\end{equation}
Under the above assumption, the population covariance matrix of $\mathcal Q_1$ is defined as
\begin{equation}\label{def_Sigma}
\Sigma : = \mathbb E \mathcal Q_1 = D^2 = \text{diag}\left(\sigma_1 ,\sigma_2 ,\ldots,\sigma_M \right). 
\end{equation}

We denote the empirical spectral density of $\Sigma$ by
\begin{equation}\label{sigma_ESD}
\pi_N := \frac{1}{M} \sum_{i=1}^M \delta_{\sigma_i}.
\end{equation}
Suppose there exists a small positive constant $\tau$ such that 
\begin{equation}\label{assm3}
\sigma_1 \le \tau^{-1} \ \ \text{ and } \ \ \pi_N([0,\tau]) \le 1 - \tau \ \ \text{for all } N.
\end{equation}
Note the first condition means the operator norm of $\Sigma$ is bounded by $\tau^{-1}$, and the second condition means that the spectrum of $\Sigma$ cannot concentrate at zero.

For definiteness, in this paper we focus on the real symmetric case, i.e. the random variable $q_{11}$ is real. However, we remark that our proof can be applied to the complex case after minor modifications if we assume in addition that $\Re\, q_{11}$ and $\Im\, q_{11}$ are independent centered random variables with variance  $1/2$. 

We summarize our basic assumptions here for future reference.
\begin{assm}\label{assm_big1}
We assume $X$ is an $M\times N$ random matrix with real i.i.d. entries satisfying (\ref{assm1}) and (\ref{assm2}). We assume $T$ is a deterministic $M\times M$ diagonal matrix satisfying (\ref{simple_assumption}) and (\ref{assm3}).  
\end{assm}

Finally, we define the following tail condition for the entries of $X$,
\begin{equation}
\lim_{s \rightarrow \infty } s^4 \mathbb{P}(\vert q_{11} \vert \geq s)=0 . \label{INTRODUCTIONEQ}
\end{equation}

\subsection{Deformed Marchenko-Pastur law}

In this paper, we will study the eigenvalue statistics of $\mathcal Q_{1,2}$ through their {\it{Green functions}} or {\it{resolvents}}.
\begin{defn}[Green functions]
We define the Green functions for $\mathcal Q_{1,2}$ as
\begin{equation}\label{def_green}
\mathcal G_1(z):=\left(DXX^*D^* -z\right)^{-1} , \ \ \ \mathcal G_2 (z):=\left(X^* D^* D X-z\right)^{-1} , \ \ \ z = E+ i\eta \in \mathbb C_+,
\end{equation}
where $\mathbb C_+$ is the upper half complex plane. We denote the empirical spectral densities (ESD) of $\mathcal Q_{1,2}$ as
$$\rho_{1}^{(N)} := \frac{1}{M} \sum_{i=1}^M \delta_{\lambda_i(\mathcal Q_1)}, \ \ \rho_{2}^{(N)} := \frac{1}{N} \sum_{i=1}^N \delta_{\lambda_i(\mathcal Q_2)}.$$
Then the Stieltjes transforms of $\rho_{1,2}$ are given by
\begin{equation}
m_1^{(N)}(z):=\int \frac{1}{x-z}\rho_{1}^{M}(dx)=\frac{1}{M} \mathrm{Tr} \, \mathcal G_1(z),\ \ m_2^{(N)}(z):=\int \frac{1}{x-z}\rho_{2}^{M}(dx)=\frac{1}{N} \mathrm{Tr} \, \mathcal G_2(z). \label{ST_m12}
\end{equation}
Throughout the following, we omit the super-index $N$ from our notations. 
\end{defn}
\begin{rem}
Since the nonzero eigenvalues of $\mathcal Q_1$ and $\mathcal Q_2$ are identical, and $\mathcal Q_1$ has $M-N$ more (or $N-M$ less) zero eigenvalues, we have
\begin{equation}
\rho_1=\rho_2 d + (1-{d})\delta_{0}, \label{def21}
\end{equation}
and
\begin{equation}\label{barm}
m_1 (z)= - \frac{1-d}{z}+d m_2 (z).
\end{equation}
\end{rem}



In the case $D= I_{M\times M}$, it is well known that the ESD of $X^* X$, $\rho_2$, converges weakly to the Marchenko-Pastur (MP) law \cite{MP} as $N\to \infty$:
\begin{equation}\label{rho_0}
\rho_{MP}(x) dx:=\frac{1}{2\pi}\frac{\sqrt{\left[(\lambda_{+}-x)(x-\lambda_{-})\right]_+}}{x} dx,
\end{equation}
where $\lambda_{\pm}=(1\pm d^{-{1}/{2}})^2$. As a result, $m_2(z)$ converges to the Stieltjes transform $m_{MP}(z)$ of $\rho_{MP}(z)$, which can be computed explicitly as
\begin{equation}
m_{MP}(z) = \frac{ d^{-1} - 1 -z + i\sqrt{(\lambda_{+}-z)(z-\lambda_{-})}}{2z}, \ \ z\in \mathbb C_+. \label{GENERMP}
\end{equation}
Moreover, one can verify that $m_{MP}(z)$ satisfies the self-consistent equation \cite{BPZ1,PY,JS}
\begin{equation}\label{MOSELF}
\frac{1}{m_{MP}(z)}= -z + d^{-1} \frac{1}{1+ m_{MP}(z)}, \ \ \Im \, m_{MP}(z) \ge 0 \text{ for } z\in \mathbb C_+.
\end{equation}
Using (\ref{def21}) and (\ref{barm}), it is also easy to get the expressions for $\rho_{1c}$ and $m_{1c}$, where $\rho_{1c}$ is the asymptotic eigenvalue density for $\mathcal{Q}_1$ and $m_{1c}$ is the associated Stieltjes transform.

If $D$ is non-identity but the ESD $\pi_N$ in (\ref{sigma_ESD}) converges weakly to some $\hat \pi$, then it was shown in \cite{MP} that the empirical eigenvalue distributions of $\mathcal Q_{1,2}$ still converge in probability to some deterministic distributions $\hat \rho_{1,2c}$, referred to as the {\it{deformed Marchenko-Pastur law}} below. They can be described through the Stieltjes transform
$$\hat m_{1,2c}(z):=\int_{\mathbb R} \frac{\hat \rho_{1,2c}(dx)}{x-z}, \ \ z = E+ i\eta \in \mathbb C_+.$$
For any given probability measure $\hat \pi$ compactly supported on $\mathbb R_+$, we define $\hat m_{2c}$ as the unique solution to the self-consistent equation \cite{BPZ1,KY2,LS}
\begin{equation}\label{deformed_MP}
\frac{1}{\hat m_{2c}(z)} = - z + d^{-1}\int\frac{x}{1+\hat m_{2c}(z) x} \hat \pi(dx), \ \ \Im \, \hat m_{2c}(z) \ge 0  \text{ for } z\in \mathbb C_+.
\end{equation}
It is well known that the functional equation (\ref{deformed_MP}) has a unique solution that is uniformly bounded on $\mathbb C_+$ under the assumptions (\ref{assm2}) and (\ref{assm3}) \cite{MP}. Letting $\eta \searrow 0$, we can recover the asymptotic eigenvalue density $\hat \rho_{2c}$ with the inverse formula
\begin{equation}\label{ST_inverse}
\hat \rho_{2c}(E) = \lim_{\eta\searrow 0} \frac{1}{\pi}\Im\, \hat m_{2c}(E+i\eta).
\end{equation}
The measure $\hat \rho_{2c}$ sometimes is called the multiplicative free convolution of $\hat \pi$ and the MP law, see e.g. \cite{AGZ,VDN}. Again, $\hat m_{1c}$ and $\hat \rho_{1c}(z)$ can be obtained easily with (\ref{def21}) and (\ref{barm}). 

Similar to (\ref{deformed_MP}), for any finite $N$ we define $m^{(N)}_{2c}$ as the unique solution to the self-consistent equation
\begin{equation}\label{deformed_MP21}
\frac{1}{m^{(N)}_{2c}(z)} = - z + d^{-1}\int\frac{x}{1+m^{(N)}_{2c}(z) x} \pi_N(dx), \ \ \Im \, m^{(N)}_{2c}(z) \ge 0  \text{ for } z\in \mathbb C_+,
\end{equation}
and define $\rho^{(N)}_{2c}$ through the inverse formula as in (\ref{ST_inverse}). Then $m^{(N)}_{1c}$ and $\rho^{(N)}_{1c}(z)$ are defined with (\ref{def21}) and (\ref{barm}). In the following, we often omit the super-index $N$ from our notations. The properties of $m_{1,2c}$ and $\rho_{1,2c}$ have been studied extensively studied; see e.g. \cite{Bai1998,Bai2006,BPZ,HHN,KY2,Silverstein1995,SC}. Here we collect some results that will be used in our proof. In particular, we need to define the rightmost edge (i.e. the {\it{soft edge}}) of $\rho_{1,2c}$.

Corresponding to the equation in (\ref{deformed_MP21}), we define the function 
\begin{equation}\label{deformed_MP2}
f(m):=- \frac{1}{m}+ d^{-1}\int\frac{x}{1 + m x} \pi_N(dx).
\end{equation}
Then $m_{2c}(z)$ can be characterized as the unique solution to the equation $z= f(m)$ with $\Im \, m\ge 0.$
\begin{lem}[Support of the deformed MP law]
The densities $\rho_{1c}$ and $\rho_{2c}$ have the same support on $\mathbb R_+$, which is a union of connected components:
\begin{equation}\label{support_rho1c}
{\rm{supp}} \, \rho_{1,2c} = \bigcup_{k=1}^p [a_{2k}, a_{2k-1}] \subset [0,\infty),
\end{equation}
where $p\in \mathbb N$ depends only on $\pi$. Here $a_k$ are characterized as following: there exists a real sequence $\{b_k\}_{k=1}^{2p}$ such that $(x,m)=(a_k, b_k)$ are the real solutions to the equations
\begin{equation}
x = f(m), \ \ \text{and} \ \ f'(m) = 0. \label{equationEm2}
\end{equation}
Moreover, we have $b_1 \in (-\sigma_1^{-1}, 0)$. Finally, under assumptions (\ref{assm2}) and (\ref{assm3}) we have $a_1 \le C$ for some positive constant $C$ depending only on $d$ and $\tau$. 
\end{lem}
For the proof of this lemma, one can refer to Lemma 2.6 and Appendix A.1 of \cite{KY2}. It is easy to observe that $m_{2c}(a_k)=b_k$ according to the definition of $f$. We shall call $a_k$ the edges of the deformed MP law $\rho_{2c}$. In particular we will focus on the rightmost edge $\lambda_r := a_1$
throughout the following. To establish our result, we need the following extra assumption.
\begin{assm}\label{assm_big2}
For $d_1$ defined in (\ref{simple_assumption}), we assume there exists a small constant $\tau>0$ such that 
\begin{equation}\label{assm_gap}
\left|1 + m_{2c}(\lambda_r) \sigma_1 \right|\ge \tau, \ \ \text{for all } N.
\end{equation}
\end{assm}

\begin{rem} 
The above assumption has previously appeared in \cite{BPZ1,DXC,Karoui,KY2}. It guarantees a regular square-root behavior of the spectral densities $\rho_{1,2c}$ near $\lambda_r$ (see Lemma \ref{lem_mbehavior} below), which is used to prove the local deformed MP law at the soft edge. Note that $f(m)$ has singularities at $m = - \sigma_i^{-1}$ for nonzero $\sigma_i$, so the condition (\ref{assm_gap}) simply rules out the singularity of $f$ at $m_{2c}(\lambda_r)=m_1$.  
\end{rem}

\subsection{Main result}

The main result of this paper is the following theorem. It establishes the necessary and sufficient condition for the edge universality of the deformed covariance matrices $\mathcal Q_{1,2}$ at the soft edge $\lambda_r$.

\begin{thm} \label{main_thm}
Let $\mathcal Q_2 = X^* T^* T X$ be an $N \times N$ sample covariance matrix with $X$ and $T$ satisfying Assumptions \ref{assm_big1} and \ref{assm_big2}. Let $\lambda_1$ be the largest eigenvalues of $\mathcal Q_2$.

\begin{itemize}

\item {\bf Sufficient condition}: If the tail condition (\ref{INTRODUCTIONEQ}) holds, then we have
\begin{equation}
\lim_{N\to \infty}\mathbb{P}(N^{{2}/{3}}(\lambda_1 - \lambda_r) \leq s) = \lim_{N\to \infty} \mathbb{P}^G(N^{{2}/{3}}(\lambda_1 - \lambda_r) \leq s), \label{SUFFICIENT}
\end{equation}
for all $s\in \mathbb R$, where $\mathbb P^G$ denotes the law for $X$ with {\it{i.i.d.}} Gaussian entries.

\item {\bf Necessary condition}: If the condition (\ref{INTRODUCTIONEQ}) does not hold for $X$, then for any fixed $s > \lambda_r $, we have
\begin{equation}
\limsup_{N \rightarrow \infty} \mathbb{P}(\lambda_1 \geq s)>0. \label{NECESSARYPART}
\end{equation}

\end{itemize}
\end{thm}

\begin{rem}
In \cite{LS}, it was proved that there exists $\gamma_0 \equiv \gamma_0(N)$ depending only on the ESD $\pi_N$ of $\Sigma$ and the aspect ratio $d$ such that
$$\lim_{N\to \infty}\mathbb{P}^G \left(\gamma_0 N^{{2}/{3}}(\lambda_1 - \lambda_r) \leq s \right) = F_1(s) $$
for all $s\in \mathbb R$, where $F_1$ is the type-1 Tracy-Widom distribution. The scaling factor $\gamma_0$ is given by \cite{Karoui}
$$\frac{1}{\gamma_0^3} = \frac{1}{d}\int \left(\frac{x}{1 + m_{2c}(\lambda_r) x}\right)^3 \pi_N(dx) - \frac{1}{m_{2c}(\lambda_r)^3},$$
and Assumption \ref{assm_big2} assures that $\gamma_0 = O(1)$ for all $N$. Hence (\ref{SUFFICIENT}) and (\ref{NECESSARYPART}) together show that the distribution of the rescaled largest eigenvalue of $\mathcal Q_2$ converges to the Tracy-Widom distribution if and only if the condition (\ref{INTRODUCTIONEQ}) holds. 
 \end{rem}
 
\begin{rem}\label{finite_correlation}
The universality result (\ref{SUFFICIENT}) can be extended to the joint distribution of the $k$ largest eigenvalues for any fixed $k$:
\begin{equation}
\lim_{N\to \infty}\mathbb{P}\left( \left(N^{{2}/{3}}(\lambda_i - \lambda_r) \leq s_i\right)_{1\le i \le k} \right) = \lim_{N\to \infty} \mathbb{P}^G\left(\left(N^{{2}/{3}}(\lambda_i - \lambda_r) \leq s_i\right)_{1\le i \le k} \right), \label{SUFFICIENT2}
\end{equation}
for all $s_1 , s_2, \ldots, s_k \in \mathbb R$. Let $H^{GOE}$ be an $N\times N$ random matrix belonging to the Gaussian orthogonal ensemble (GOE). The joint distribution of the $k$ largest eigenvalues of $H^{GOE}$, $\mu^{GOE}_1 \ge \ldots \ge \mu_k^{GOE}$, can be written in terms of the Airy kernel for any fixed $k$ \cite{Forr}. It was proved in \cite{LS} that
$$ \lim_{N\to \infty} \mathbb{P}^G \left(\left(\gamma_0 N^{{2}/{3}}(\lambda_i - \lambda_r) \leq s_i\right)_{1\le i \le k} \right) = \lim_{N\to \infty} \mathbb{P}\left(\left( N^{{2}/{3}}(\mu_i^{GOE} - 2) \leq s_i\right)_{1\le i \le k} \right), $$
for all $s_1 , s_2, \ldots, s_k \in \mathbb R$. Hence (\ref{SUFFICIENT2}) gives a complete description of the finite-dimensional correlation functions of the extremal eigenvalues of $\mathcal Q_2$.
 \end{rem}

\section{Basic notations and tools}\label{sec_maintools}

\subsection{Notations}

Following the notations in \cite{EKYY,EKYY1}, we will use the following definition to characterize events of high probability.

\begin{defn}[High probability event] \label{high_prob}
Define
\begin{equation}\label{def_phi}
\varphi:=(\log N)^{\log \log N}.
\end{equation} 
We say that an $N$-dependent event $\Omega$ holds with $\xi$-high probability if there exists constant $c,C>0$ independent of $N$, such that
\begin{equation}
\mathbb{P}(\Omega) \geq 1-N^{C} \exp(- c\varphi^{\xi}),  \label{D25}
\end{equation}
for all sufficiently large $N$. For simplicity, for the case $\xi=1$, we just say high probability. Note that if $\Omega$ holds with $\xi$-high probability, then $\mathbb P(\Omega) \ge 1 - \exp(-c'\varphi^{\xi})$ for any $0\le c' <c$. 
\end{defn}

\begin{defn}[Bounded support condition] \label{defn_support}
A family of $M\times N$ matrices $X =(x_{ij})$ are said to satisfy the bounded support condition with $q\equiv q(N)$ if
\begin{equation}
\mathbb{P}\left(\max_{1\le i \le M,1\le j \le N}\vert x_{ij}\vert \le q\right) \geq 1-e^{-N^c}, \label{eq_support}
\end{equation}
for some $c>0$. 
Here $q\equiv q(N)$ depends on $N$ and usually satisfies
$$ N^{-{1}/{2}}\log N \leq q \leq N^{- \phi}, $$
for some small positive constant $\phi$. Whenever (\ref{eq_support}) holds, we say that $x_{ij}$ has support $q$.
\end{defn}
\begin{rem}
Note that the Gaussian distribution satisfies the condition (\ref{eq_support}) with $q< N^{-\phi}$ for any $\phi<1/2$. We also remark that by Definition \ref{high_prob}, the event $\left\{\vert x_{ij}\vert \le q, \forall 1\le i \le M,1\le j \le N\right\}$ in (\ref{eq_support}) holds with $\xi$-high probability for any constant $\xi>0$. For this reason, the bad event $\left\{\vert x_{ij}\vert \ge q \text{ for some }i,j\right\}$ is negligible, and we will not consider the case it happens throughout the proof.
\end{rem}

Next we introduce a convenient self-adjoint linearization trick, which has been proved to be useful in studying the local laws of deformed sample random matrices \cite{DXC, KY2,XYY}. We define the following $(N+M)\times (N+M)$ block matrix, which is a linear function of $X$.
\begin{defn}[Linearizing block matrix]\label{def_linearHG}
For $z\in \mathbb C_+ $, we define the $(N+M)\times (N+M)$ self-adjoint matrices
 \begin{equation}\label{linearize_block}
   H \equiv H(X): = \left( {\begin{array}{*{20}c}
   { 0 } & DX  \\
   {(DX)^*} & {0}  \\
   \end{array}} \right),
 \end{equation}
and
 \begin{equation}\label{eqn_defG}
 G \equiv G (X,z):= \left( {\begin{array}{*{20}c}
   { - I_{M\times M}} & DX  \\
   {(DX)^*} & { - zI_{N\times N}}  \\
\end{array}} \right)^{-1}.
 \end{equation}
\end{defn}
\begin{defn}[Index sets]\label{def_index} 
We define the index sets
\[\mathcal I_1:=\{1,...,M\}, \ \ \mathcal I_2:=\{M+1,...,M+N\}, \ \ \mathcal I:=\mathcal I_1\cup\mathcal I_2.\]
Then we label the indices of the matrices according to 
$$X= (X_{i\mu}:i\in \mathcal I_1, \mu \in \mathcal I_2) \ \ \text{and} \ \ D={\rm{diag}}(D_{ii}: i\in \mathcal I_1).$$  
In the following, whenever referring to the entries of $H$ and $G$, we will consistently use the latin letters $i,j\in\mathcal I_1$, greek letters $\mu,\nu\in\mathcal I_2$, and $a,b\in\mathcal I$. For $1\le i \le \min\{N,M\}$ and $M+1 \le \mu  \le M+\min\{N,M\}$, we introduce the notations $\bar i:=i+M \in \mathcal I_2$ and $\bar\mu:=\mu-M \in \mathcal I_1$. For any $\mathcal I \times \mathcal I$ matrix $A$, we denote the $2\times 2$ submatrices $A_{[ij]}$ as
\begin{equation}\label{Aij_group}
A_{[ij]}=\left( {\begin{array}{*{20}c}
   {A_{ij} } & {A_{i\bar j} }  \\
   {A_{\bar i j} } & {A_{\bar i\bar j} }  \\
\end{array}} \right), \ \ 1\le i,j \le \min\{N,M\}.
\end{equation}
We shall call $A_{[ij]}$ a diagonal group if $i=j$, and an off-diagonal group otherwise .
\end{defn}

It is easy to verify that the eigenvalues $\lambda_1(H)\ge \ldots \ge \lambda_{M+N}(H)$ of $H$ are related to the ones of $\mathcal Q_1$ through
\begin{equation}\label{Heigen}
\lambda_i(H)=-\lambda_{N+M-i+1}(H)=\sqrt{\lambda_i\left(\mathcal Q_2\right)}, \ \ 1\le i \le N\wedge M,
\end{equation}
and
$$\lambda_i(H)=0, \ \ N\wedge M + 1 \le i \le N\vee M,$$
where we used the notations $N\wedge M:=\min\{N,M\}$ and $N\vee M:=\max\{N,M\}$. Furthermore, by Schur complement formula, we can verify that
\begin{align} 
G & = \left( {\begin{array}{*{20}c}
   { z(DXX^*D^*-z)^{-1}} &  (DXX^*D^*-z)^{-1}DX \\
   {X^*D^*(DXX^*D^*-z)^{-1}} & { (X^*D^* D X-z)^{-1}}  \\
\end{array}} \right) \nonumber\\
&= \left( {\begin{array}{*{20}c}
   { z\mathcal G_1} & \mathcal G_1 DX  \\
   {X^*D^* \mathcal G_1} & { \mathcal G_2 }  \\
\end{array}} \right) = \left( {\begin{array}{*{20}c}
   { z\mathcal G_1} & DX\mathcal G_2   \\
   {\mathcal G_2}X^*D^* & { \mathcal G_2 }  \\ 
\end{array}} \right). \label{green2}
\end{align}
Thus a control of $G$ yields directly a control of the resolvents $\mathcal G_{1,2}$ defined in (\ref{def_green}). By (\ref{green2}), we immediately get that
$$m_1=\frac{1}{Mz}\sum_{i\in \mathcal I_1}G_{ii}, \ \ m_2=\frac{1}{N}\sum_{\mu \in \mathcal I_2}G_{\mu\mu}.$$
 
Next we introduce the spectral decomposition of $G$. Let
$$DX = \sum\limits_{k = 1}^{N\wedge M} {\sqrt {\lambda_k} \xi_k } \zeta _{k}^* ,$$
be the singular value decomposition of $DX$, where
$$\lambda_1\ge \lambda_2 \ge \ldots \ge \lambda_{N\wedge M} \ge 0 = \lambda_{N\wedge M+1} = \ldots = \lambda_{N\vee M},$$
and $\{\xi_{k}\}_{k=1}^{M}$ and $\{\zeta_{k}\}_{k=1}^{N}$ are orthonormal bases of $\mathbb R^{\mathcal I_1}$ and $\mathbb R^{\mathcal I_2}$, respectively. Then using (\ref{green2}), we can get that for $i,j\in \mathcal I_1$ and $\mu,\nu\in \mathcal I_2$,
\begin{align}
G_{ij} = \sum\limits_{k = 1}^{M} \frac{z\xi_k(i) \xi_k^*(j)}{\lambda_k-z},\ \quad \ &G_{\mu\nu} = \sum\limits_{k = 1}^{N} \frac{\zeta_k(\mu) \zeta_k^*(\nu)}{\lambda_k-z}, \label{spectral1}\\
G_{i\mu} = \sum\limits_{k = 1}^{N\wedge M} \frac{\sqrt{\lambda_k}\xi_k(i) \zeta_k^*(\mu)}{\lambda_k-z}, \ \quad \ &G_{\mu i} = \sum\limits_{k = 1}^{N\wedge M} \frac{\sqrt{\lambda_k}\zeta_k(\mu) \xi_k^*(i)}{\lambda_k-z}.\label{spectral2}
\end{align}

\subsection{Main tools}\label{sec_tools}

For small constant $c_0>0$ and large constants $C_0, C_1 >0$, we define the domain of the spectral parameter $z=E+i\eta$ by
\begin{equation}
S(c_0,C_0,C_1)= \left\{z=E+i \eta \in \mathbb{C}_+: \lambda_r - c_0 \leq E \leq C_0 \lambda_r, \ \varphi^{C_1} N^{-1} \leq \eta \leq 1 \right\}. \label{SSET1}
\end{equation}
We define the distance to the rightmost edge as
\begin{equation}
 \kappa := \vert E -\lambda_r\vert , \ \ z= E+i\eta.\label{KAPPA}
\end{equation}
Then we have the following lemma, which summarizes some basic behaviors of $m_{1,2c}$ and $\rho_{1,2c}$.

\begin{lem}[Lemmas 2.1 and 2.3 in \cite{BPZ}]\label{lem_mbehavior}
There exists sufficiently small constant $\tilde c>0$ such that 
\begin{equation}
\rho_{2c}(x) \sim \sqrt{\lambda_r-x}, \ \ \text{ for } x \in \left[\lambda_r - 2\tilde c,\lambda_r \right].\label{SQUAREROOT}
\end{equation}
The Stieltjes transforms $m_{2c}$ satisfy that
\begin{equation}\label{Immc}
\vert m_{2c}(z) \vert \sim 1, 
\end{equation}
and
\begin{equation}
  \operatorname{Im} m_{2c}(z) \sim \begin{cases}
    {\eta}/{\sqrt{\kappa+\eta}}, & E\geq \lambda_r \\
    \sqrt{\kappa+\eta}, & E \le \lambda_r\\
  \end{cases},  \label{SQUAREROOTBEHAVIOR}
\end{equation}
for $z = E+i\eta\in S(\tilde c,C_0,0).$
\end{lem}

\begin{rem}
Recall that $a_k$ are the edges of the spectral density $\rho_{2c}$; see (\ref{support_rho1c}). Hence $\rho_{2c}(a_k)=0$, and we must have $a_k < \lambda_r - 2\tilde c$ for $2\le k \le 2p$. In particular, $S(c_0,C_0,C_1)$ is away from all the other edges if we choose $c_0 \le \tilde c$. 
\end{rem}

\begin{defn} [Classical locations of eigenvalues]
The classical location $\gamma_j$ of the $j$-th eigenvalue of $\mathcal Q_2$ is defined as
\begin{equation}\label{gammaj}
\gamma_j:=\sup_{x}\left\{\int_{x}^{+\infty} \rho_{2c}(x)dx > \frac{j-1}{N}\right\}.
\end{equation}
\end{defn}
\begin{rem}
If $\gamma_j$ lies in the bulk of some component of $\rho_{2c}$, then by the continuity of $\rho_{2c}$ we can define $\gamma_j$ through
\begin{equation*}
\int_{\gamma_j}^{+\infty} \rho_{2c}(x)dx = \frac{j-1}{N}.
\end{equation*}
We can also define the classical location of the $j$-th eigenvalue of $\mathcal Q_1$ by changing $\rho_{2c}$ to $\rho_{1c}$ and $(j-1)/{N}$ to $(j-1)/{M}$ in (\ref{gammaj}). By (\ref{def21}), it gives the same location as $\gamma_j$ for $j\le N\wedge M$.
\end{rem}

\begin{defn}[Deterministic limit of $G$]
We define the deterministic limit $\Pi$ of the Green function $G$ in (\ref{green2}) as
\begin{equation}
\Pi (z): = \left( {\begin{array}{*{20}c}
   { -\left(1+m_{2c}(z)\Sigma \right)^{-1} } & 0  \\
   0 & {m_{2c}(z)I_{N\times N}}  \\
\end{array}} \right) ,
\end{equation}
where $\Sigma$ is defined in (\ref{def_Sigma}).
\end{defn}

In the rest of this section, we introduce some results that will be used in the proof of Theorem \ref{main_thm} in Section \ref{sec_cutoff}. Their proofs will be given in subsequent sections.

\begin{lem}
\label{lem_small}

Suppose Assumption \ref{assm_big1} holds and $X$ satisfies the bounded support condition (\ref{eq_support}) for some $q\le N^{-\phi}$ with $\phi$ being any positive constant. Let $c_1>0$ be sufficiently small and fix $C_0>0$. Then there exist constants $C_1,  \Lambda>0$ and $\xi_1 \ge 3$ such that the following results hold with $\xi_1$-high probability:
\begin{itemize}
\item[(1)] {\bf Local deformed MP law}:
\begin{equation}
\bigcap_{z \in S(2c_1,C_0,C_1) } \left\{ \vert m_2(z)-m_{2c}(z) \vert \leq \varphi^{C_1}\left(\min \left\{q,\frac{q^2}{\sqrt{\kappa+\eta}} \right\}+\frac{1}{N\eta} \right) \right\}, \label{MPBOUNDS}
\end{equation}
\begin{equation}\label{DIAGONAL}
\bigcap_{z \in S(2c_1,C_0,C_1)} \left\{ \max_{a,b \in \mathcal I} \vert G_{ab}(z)- \Pi_{ab}(z) \vert \leq \varphi^{ C_1}\left(q+ \sqrt{\frac{\operatorname{Im} m_{2c}(z) }{N \eta}}+ \frac{1}{N\eta}\right) \right\};
\end{equation}

\item[(2)] {\bf Bound on $\|H\|$}:
\begin{equation}
\|H\|\le  \Lambda ; \label{boundH}
\end{equation}

\item[(3)]{\bf Delocalization}:
\begin{equation}
\max_{k, i}|\xi_k(i)|^2+\max_{k, \mu}|\zeta_k(\mu)|^2 \le \frac{\varphi^{C_1}}{N}. \label{delocal}
\end{equation}
\end{itemize}

Furthermore if $q\le N^{-\phi}$ for some constant $\phi > {1}/{3}$, then the following rigidity result holds with $\xi_1$-high probability:
\begin{itemize}
\item[(4)] {\bf Rigidity of eigenvalues}:
\begin{equation}
\bigcap_{j: \lambda_r - c_1 \le \gamma_j \le \lambda_r}\left\{ \vert \lambda_j - \gamma_j \vert \leq \varphi^{C_1}\left( j^{-1/3}N^{-2/3} + q^2 \right) \right\}.\label{SEC}
\end{equation}
\end{itemize}
\end{lem}

\begin{lem}[Edge universality: small support case]\label{lem_smallcomp}
Let $X^W$ and $X^V$ be any two {\it{i.i.d.}} sample covariance matrices satisfying Assumption \ref{assm_big1} and the bounded support condition (\ref{eq_support}) for some $q\le N^{-\phi}$ with ${1}/{3} < \phi \leq {1}/{2}$. Then there exist constants $\epsilon,\delta >0$ such that, for any $s\in \mathbb R$, we have
\begin{equation} 
\mathbb{P}^V (N^{\frac{2}{3}}(\lambda_1-\lambda_{r}) \leq s-N^{-\epsilon})-N^{-\delta} \leq
\mathbb{P}^W(N^{\frac{2}{3}}(\lambda_1-\lambda_{r})\leq s) \leq \mathbb{P}^V(N^{\frac{2}{3}}(\lambda_1-\lambda_{r}) \leq s+N^{-\epsilon})+N^{-\delta}, \label{EDDDD}
\end{equation}
where $\mathbb{P}^V$ and $\mathbb{P}^W$ denote the laws of $X^V$ and $X^W$, respectively.
\end{lem}
\begin{rem} \label{rigid_multi}
As in \cite{EKYY,EYY,LY}, Lemma \ref{lem_smallcomp}, as well as Theorem \ref{lem_comparison} below, can be can be generalized to finite correlation functions of the $k$ largest eigenvalues for any fixed $k$:
\begin{align}
& \mathbb{P}^V \left( \left(N^{{2}/{3}}(\lambda_i-\lambda_{r}) \leq s_i -N^{-\epsilon}\right)_{1\le i \le k}\right)-N^{-\delta} \le \mathbb{P}^W \left( \left(N^{{2}/{3}}(\lambda_i-\lambda_{r}) \leq s_i \right)_{1\le i \le k}\right) \nonumber\\
& \leq \mathbb{P}^V \left( \left(N^{{2}/{3}}(\lambda_i-\lambda_{r}) \leq s_i + N^{-\epsilon}\right)_{1\le i \le k}\right)+ N^{-\delta} \label{EDDDD_ext}
\end{align}
for sufficiently large $N$. The proof of (\ref{EDDDD_ext}) is similar to that of (\ref{EDDDD}) except that it uses a general form of the Green function comparison theorem; see e.g. \cite[Theorem 6.4]{EYY}. As a corollary, we can then prove the stronger universality result (\ref{SUFFICIENT2}).
\end{rem}

In fact under different assumptions, Lemma \ref{lem_small} has been proved previously in \cite{BPZ1,KY2} in slightly different forms. For completeness, we will give a brief proof for it in Appendix \ref{appendix1} under our assumptions. Then with Lemma \ref{lem_small}, Lemma \ref{lem_smallcomp} follows from a routine application of the Green function comparison theorem; we refer the reader to Section 4 of \cite{BPZ1} and Section 10 of \cite{KY2}.


For any matrix $X$ satisfying Assumption \ref{assm_big1} and the tail condition (\ref{INTRODUCTIONEQ}), we can construct a matrix $X_1$ that approximates $X$ with high probability and satisfies Assumption \ref{assm_big1}, the bounded support condition (\ref{eq_support}) with $q\le N^{-\phi}$ for some small $0<\phi<1/3$, and
\begin{equation}\label{conditionA3}
\mathbb{E}\vert x_{ij} \vert^3 \leq B N^{-{3}/{2}}, \   \mathbb{E} \vert x_{ij} \vert^4  \leq B (\log N)N^{-2}.
\end{equation}
for some constant $B>0$ (see Section \ref{sec_cutoff}, proof of the sufficient condition). We will need the following improved local deformed MP law and eigenvalues rigidity result for matrices with large support and satisfying condition (\ref{conditionA3}).



\begin{thm}[Rigidity of eigenvalues: large support case] \label{thm_largerigidity}
Suppose $X$ satisfies Assumption \ref{assm_big1}, the bounded support condition (\ref{eq_support}) with $q\le N^{-\phi}$ for some constant $\phi>0$, and condition (\ref{conditionA3}). Fix the constants $c_1$, $C_0$, $C_1$, $A$ and $\xi_1$ as given in Lemma \ref{lem_small}. Then there exists constant $C_2>0$, depending only on $c_1$, $B$ and $\phi$, such that with high probability we have
\begin{equation}
\max_{z\in S(c_1, C_0, C_2)} \vert m_2(z)-m_{2c}(z) \vert \leq \frac{\varphi^{C_2}}{N \eta}, \label{NEWMPBOUNDS}
\end{equation}
for sufficiently large $N$. Moreover, (\ref{NEWMPBOUNDS}) implies that with high probability the following rigidity results hold for some $C_\xi>0$:
\begin{equation}\label{P1P1P1P1P1}
\bigcap_{j: \lambda_r - c_1 \le \gamma_j \le \lambda_r}\left\{ \vert \lambda_j - \gamma_j \vert \leq \varphi^{C_\xi}  j^{-1/3}N^{-2/3} \right\}
\end{equation}
and
\begin{equation}
\sup_{ \lambda_r - c_1 \le E \leq C_0 \lambda_r} \vert n(E)-n_{c}(E) \vert \leq \frac{\varphi^{C_\xi}}{N},  \label{PPPPPP}
\end{equation}
where
$$n(E):=\frac{1}{N} \# \{ \lambda_j \ge E\}, \ \ n_{c}(E):=\int^{+\infty}_E \rho_{2c}(x)dx.$$
\end{thm}

\begin{thm}\label{lem_comparison}
Let $X^W$ and $X^V$ be any two {\it{i.i.d.}} sample covariance matrices satisfying the assumptions in Theorem \ref{thm_largerigidity}. Then there exist constants $\epsilon,\delta >0$ such that, for any $s\in \mathbb R$, we have
\begin{equation} 
\mathbb{P}^V (N^{\frac{2}{3}}(\lambda_1-\lambda_{r}) \leq s-N^{-\epsilon})-N^{-\delta} \leq
\mathbb{P}^W(N^{\frac{2}{3}}(\lambda_1-\lambda_{r})\leq s) \leq \mathbb{P}^V(N^{\frac{2}{3}}(\lambda_1-\lambda_{r}) \leq s+N^{-\epsilon})+N^{-\delta}, 
\end{equation}
where $\mathbb{P}^V$ and $\mathbb{P}^W$ denote the laws of $X^V$ and $X^W$, respectively.
\end{thm}

\begin{lem}[Bounds on $G_{ij}$: large support case] \label{thm_largebound}
Let $X$ be a matrix satisfying the assumptions in Theorem \ref{thm_largerigidity}. Then for any $0<c<1$ and $z\in S(c_1,C_0,C_2)\cap \{z= E+i\eta: \eta \geq N^{-1+c}\}$,  we have the following weak bound
\begin{equation}\label{weak_off}
\mathbb{E} \vert G_{ab}(z) \vert^2 \leq \varphi^{C_3}\left(\frac{\operatorname{Im} m_{2c}(z)}{N  \eta}+\frac{1}{(N \eta)^2}\right), \ a \neq b,
\end{equation}
for some large enough constant $C_3>0$.
\end{lem}

In proving Theorem \ref{thm_largerigidity} and Lemma \ref{thm_largebound}, we will make use of the results for small support matrices in Lemma \ref{lem_small} and Lemma \ref{thm_largebound}. In fact, given any matrix $X$ satisfying the assumptions in Theorem \ref{thm_largerigidity}, we can construct a matrix $\tilde X$ having the same first four moments as $X$ but with small support $q=\mathcal O(N^{-1/2}\log N)$.

\begin{lem} [Lemma 5.1 in \cite{LY}]\label{lem_decrease}
Suppose $X=(x_{ij})$ satisfies the assumptions in Theorem \ref{thm_largerigidity}. Then there exists another matrix $\tilde{X}=(\tilde x_{ij})$, such that $\tilde{X}$ satisfies the bounded support condition (\ref{eq_support}) with $q=\mathcal O(N^{-1/2}\log N)$, and the first four moments of the entries of $X$ and $\tilde{X}$ match, i.e.
\begin{equation}\label{match_moments}
\mathbb Ex_{ij}^k =\mathbb E\tilde x_{ij}^4, \ \ k=1,2,3,4.
\end{equation}
\end{lem}

From Lemma \ref{lem_small}, we can get that Theorems \ref{thm_largerigidity} and \ref{thm_largebound} hold for $\tilde X$. Then due to (\ref{match_moments}), we expect that $X$ has ``similar properties" as $\tilde X$, so that Theorems \ref{thm_largerigidity} and \ref{thm_largebound} also hold for $X$. This will be proved with a Green function comparison method: we expand the Green functions with $X$ in terms of Green functions with $\tilde X$ using resolvent expansions and estimate the high order errors; see Section \ref{comparison} for more details.

\section{Proof of of the main result}\label{sec_cutoff}

In this section, we prove Theorem \ref{main_thm} with the results in Section \ref{sec_tools}. We begin by proving the necessary condition.

\begin{proof}[Proof of the Necessary condition]
Assume that $\lim_{s \rightarrow \infty } s^4 \mathbb{P}(\vert q_{11} \vert \geq s)\ne 0$. Then we can find a constant $0<c_0<{1}/{2}$ and a sequence $\{r_n\}$ such that $r_n \rightarrow \infty$ as $n \rightarrow \infty$ and
\begin{equation}
\mathbb{P}(\vert q_{ij} \vert \geq r_n) \geq c_0 r_n^{-4}. \label{KeyLemma}
\end{equation}
Fix any $s>\lambda_r$. We denote $L:=\left\lfloor \tau M \right\rfloor$, $I:=\sqrt{\tau^{-1}s}$ and define the event
$$\Gamma_N=\left\{ \text{There exist} \ i \text{ and } j, \ 1 \leq i  \le L, 1\le j \leq N, \  \text{such that} \ \vert x_{ij} \vert \geq I \right\} .$$ 

We first show that $\lambda_1(\mathcal Q_2) \ge s$ when $\Gamma_N$ holds. Suppose $\vert x_{ij} \vert \geq \sqrt{\tau^{-1}s}$ for some $1 \leq i  \le L$ and $1\le j \leq N$. Let $\mathbf u \in \mathbb{R}^N$ such that $\mathbf u(k)=\delta_{kj}$. By assumption (\ref{assm3}), we have $\sigma_i \ge \tau$ for $i \le L$.  Hence
\begin{equation*}
\lambda_1(\mathcal Q_2) \geq \langle \mathbf u, (TX)^{*} (TX) \mathbf u \rangle = \sum_{k=1}^M \sigma_k x_{kj}^2  \geq  \sigma_i x_{ij}^2 \ge \tau \left( \sqrt{\tau^{-1}s}\right)^2 = s.
\end{equation*}
Now we choose $N\in \left\{ \left\lceil ({r_n}/ I)^2 \right\rceil : n \in \mathbb N\right\}$.  With the choice $N=\left\lceil ({r_n}/ I)^2 \right\rceil$,  we have
\begin{align}
1-\mathbb{P}({\Gamma}_{N})& = \left(1-\mathbb{P}(\vert x_{ij} \vert \geq I) \right)^{NL} = \left(1-\mathbb{P}(\vert q_{ij} \vert \geq r_n)\right)^{NL}\nonumber \\
& \leq \left(1-c_0 r_n^{-4}\right)^{NL}  \leq (1-c_1N^{-2})^{c_2 N^2} ,\label{BOUND1}
\end{align}
for some constant $c_1 > 0$ depending on $c_0$ and $I$ and some constant $c_2$ depending on $\tau$ and $d$. Since $(1-c_1N^{-2})^{c_2 N^2} \le  c_3 $ for some constant $0< c_3 <1$ independent of $N$, the above inequality shows that $\mathbb{P}({\Gamma}_{N}) \ge 1- c_3 > 0$. 
This shows that $\limsup_{N \rightarrow \infty} \mathbb{P}(\Gamma_N)>0$ and concludes the proof.
\end{proof}


\begin{proof}[Proof of the Sufficient condition]
Given the matrix $X$ satisfying Assumption \ref{assm_big1} and the tail condition (\ref{INTRODUCTIONEQ}), we introduce a cutoff on its matrix entries at the level $N^{-\epsilon}$. For any fixed $\epsilon>0$, define
\begin{equation*}
\alpha_N:=\mathbb{P}\left(\vert q_{11} \vert > N^{{1}/{2}-\epsilon}\right), \ \ \beta_N:=\mathbb{E}\left[\mathbf{1}{\left(|q_{11}|> N^{{1}/{2}-\epsilon}\right)}q_{11} \right].
\end{equation*}
By (\ref{INTRODUCTIONEQ}) and integration by parts, we have that for any $\delta>0$ and large enough $N$,
\begin{equation}
\alpha_N \leq \delta N^{-2+4\epsilon}, \ \ \vert \beta_N \vert \leq \delta N^{-{3}/{2}+3 \epsilon} . \label{BBBOUNDS}
\end{equation}

Let $\rho(x)$ be the distribution density of $q_{11}$. Then we define independent random variables $q_{ij}^s$, $q_{ij}^l$, $c_{ij}$, $1\le i \le M$ and $1\le j \le N$ in the following ways:
\begin{itemize}
\item $q_{ij}^s$ has distribution density $\rho_s(x)$, where
\begin{equation}
\rho_s(x)= \mathbf{1}\left(\left| x-\frac{\beta_N}{1-\alpha_N}  \right| \leq N^{\frac{1}{2}-\epsilon} \right) \frac{\rho\left(x-\frac{\beta_N}{1-\alpha_N}\right)}{1-\alpha_N}; \label{DEF1}
\end{equation}
\item $q_{ij}^l$ has distribution density $\rho_l(x)$, where
\begin{equation}
\rho_l(x)= \mathbf{1}\left(\left| x-\frac{\beta_N}{1-\alpha_N}  \right| > N^{\frac{1}{2}-\epsilon} \right)\frac{\rho\left(x-\frac{\beta_N}{1-\alpha_N}\right)}{\alpha_N}; \label{DEF2}
\end{equation}
\item $c_{ij}$ is a Bernoulli 0-1 random variable with $\mathbb{P}(c_{ij}=1)=\alpha_N$ and $\mathbb{P}(c_{ij}=0)=1-\alpha_N$.
\end{itemize}
Let $X^s$, $X^l$ and $X^c$ be random matrices such that $X^s_{ij} = N^{-1/2}q_{ij}^s$, $X^l_{ij} = N^{-1/2}q_{ij}^l$ and $X^c_{ij} = c_{ij}$.
By (\ref{DEF1}), (\ref{DEF2}) and the fact that $X^c_{ij}$ is Bernoulli, it is easy to check that for independent $X^s$, $X^l$ and $X^c$,
\begin{equation}
X_{ij} \stackrel{d}{=} X^s_{ij}\left(1-X^c_{ij}\right)+X^l_{ij}X^c_{ij} - \frac{1}{\sqrt{N}}\frac{\beta_N}{1-\alpha_N}, \label{T3}
\end{equation}
where by (\ref{BBBOUNDS}), we have
\begin{equation*}
\left\vert \frac{1}{\sqrt{N}}\frac{\beta_N}{1-\alpha_N} \right\vert \leq 2\delta N^{-2+3\epsilon}.
\end{equation*}
Therefore, if we define the $M\times N$ matrix $Y=(Y_{ij})$ by 
$$Y_{ij}=\frac{1}{\sqrt{N}}\frac{\beta}{1-\alpha_N} \ \ \text{ for all }i\text{ and }j,$$
 we have $ \| Y \| \leq c N^{-1+3\epsilon} $ for some constant $c>0$ depending on $\delta$ and $d$. Using the bound (\ref{boundH}), it is easy to see that  
\begin{align}\label{const_err}
\left|\lambda\left((X+Y)^*D^* D(X+Y)\right) - \lambda\left(X^* D^* DX\right)\right| = O\left( N^{-1+3\epsilon} \right),
\end{align}
with $\xi_1$-high probability. Hence the deterministic part in (\ref{T3}) is negligible under the scaling $N^{2/3}$.

By (\ref{INTRODUCTIONEQ}) and integration by parts, it is easy to check that
\begin{equation}
\mathbb{E} q^{s}_{ij} =0, \ \ \mathbb{E}\vert q^s_{ij}\vert^2=1-O(N^{-1+2 \epsilon}), \ \ \mathbb{E}\vert q^s_{ij}\vert^3 = O(1), \  \ \mathbb{E}\vert q^s_{ij} \vert^4=O(\log N).\label{estimate_qs}
\end{equation}
We note that $X_1:=(\mathbb{E}\vert q^s_{ij} \vert^2)^{-{1}/{2}}X^s$ is a matrix that satisfies the assumptions for $X$ in Theorem \ref{lem_comparison}. 
Together with the estimate for $\mathbb{E}\vert q^s_{ij}\vert^2$ in (\ref{estimate_qs}), we conclude that there exists constants $\epsilon,\delta>0$ such that for any $s\in \mathbb R$,
\begin{equation}\label{univ_small}
\mathbb{P}^G (N^{\frac{2}{3}}(\lambda_1-\lambda_{r}) \leq s-N^{-\epsilon})-N^{-\delta} \leq
\mathbb{P}^s(N^{\frac{2}{3}}(\lambda_1-\lambda_{r})\leq s) \leq \mathbb{P}^G(N^{\frac{2}{3}}(\lambda_1-\lambda_{r}) \leq s+N^{-\epsilon})+N^{-\delta},
\end{equation}
where $\mathbb P^s$ is the law for $X^s$ and $\mathbb P^G$ is the law for a Gaussian covariance matrix. Now we write the first two terms on the right-hand side of (\ref{T3}) as
$$X^s_{ij}(1-X^c_{ij})+X^l_{ij}X^c_{ij} = X^s_{ij} + R_{ij} X^c_{ij},$$
where $ R_{ij}:=X^l_{ij}-X^s_{ij}.$ It remains to show that the effect of the $R_{ij} X^c_{ij}$ terms on $\lambda_1$ is negligible. We call the corresponding matrix as $R^c :=(R_{ij} X^c_{ij})$.

Note that $X^c_{ij}$ is independent of $X^s_{ij}$ and $R_{ij}$. We first introduce a cutoff on matrix $X^c$ as $\tilde X^c : =\mathbf{1}_{A}X^c$, where
\begin{equation*}
A:=\left\{\#\{(i,j):X^c_{ij}=1\}\leq N^{5\epsilon}\right\} \cap \left\{X^c_{ij}=X^c_{kl}=1  \Leftrightarrow \{i,j\}=\{k,l\} \ \text{or} \  \{i,j\} \cap \{k,l\}=\emptyset \right\}.
\end{equation*}
If we regard the matrix $X^c$ as a sequence $\mathbf X^c$ of $NM$ {\it{i.i.d.}} Bernoulli random variables, it is easy to obtain from the large deviation formula that
\begin{equation}\label{LDP_B}
\mathbb{P}\left(\sum_{i=1}^{MN} \mathbf X^c_i \leq N^{5 \epsilon}\right) \geq 1- \exp(- N^{\epsilon}),
\end{equation}
for sufficiently large $N$. Suppose the number $m$ of the nonzero elements in $X^c$ is given. Then it is easy to check that
\begin{equation}\label{LDP_C}
\mathbb P\left(\exists \, i = k, j\ne l \ \text{or} \  i\ne k, j =l \text{ such that } X^c_{ij}=X^c_{kl}=1 \left| \sum_{i=1}^{MN} \mathbf X^c_i = m \le N^{5\epsilon}\right. \right) \ge 1- O(m^2N^{-1}).
\end{equation}
Combining the estimates (\ref{LDP_B}) and (\ref{LDP_C}), we get that 
\begin{equation}\label{prob_A}
\mathbb P(A) \ge 1 -  O(N^{-1+10\epsilon}).
\end{equation}
On the other hand, by condition (\ref{INTRODUCTIONEQ}), we have
\begin{equation}\label{prob_R}
\mathbb{P}\left(|R_{ij}| \geq \omega \right) \leq \mathbb{P}\left(|q_{ij}| \geq \frac{\omega}{2}N^{1/2}\right) \leq o(N^{-2}) ,
\end{equation}
for any fixed constant $\omega>0$. Hence if we introduce the matrix 
$$E = \mathbf 1\left(A\cap \left\{\max_{i,j} |R_{ij}| \leq \omega \right\}\right) R^c,$$
then 
\begin{equation}\label{EneR}
\mathbb P( E = R^c)=1-o(1),
\end{equation}
by (\ref{prob_A}) and (\ref{prob_R}). Thus we only need to study the largest eigenvalue of $(X^s + E)^* D^* D (X^s +E),$ where $\max_{i,j} |E_{ij}| \le \omega$ and the rank of $E$ is less than $N^{5\epsilon}$.

We only need to prove that
\begin{equation}\label{eq_claim}
\mathbb P\left(\left|\lambda_1^s - \lambda_1^E\right| \le N^{-3/4}\right) = 1-o(1),
\end{equation}
where $\lambda_1^s:=\lambda_1\left((X^s)^* D^* DX^s \right)$ and $ \lambda_1^E:=\lambda_1\left((X^s+E)^*D^* D(X^s+E) \right)$. In fact, the estimate (\ref{eq_claim}), combined with (\ref{const_err}), (\ref{univ_small}) and (\ref{EneR}), concludes (\ref{SUFFICIENT}).

Now we prove (\ref{eq_claim}). Recall that $ \tilde{X}^c $ is independent of $X^s$, so the position of the nonzero elements of $E$ is independent of $X^s$. By symmetry, we can assume the $s$ nonzero entries of $E$ are exactly
\begin{equation}
E_{11}, E_{22}, \cdots,  E_{ss}. \label{STRUCTURE}
\end{equation}
Now we define the matrices
$$  H^s : = \left( {\begin{array}{*{20}c}
   { 0 } & DX^s  \\
   {(DX^s)^*} & {0}  \\
   \end{array}} \right) \ \ \text{ and } \ \  H^{E} : = H^s + P, \ \ P:= \left( {\begin{array}{*{20}c}
   { 0 } & DE  \\
   { (DE)^*} & {0}  \\
   \end{array}} \right).$$
Then we have the eigendecomposition $P = V P_DV^*,$ where $P_D$ is a $2s \times 2s$ diagonal matrix
$$P_D= \text{diag}\left(\sqrt{\sigma_{1}}E_{11}, \ldots, \sqrt{\sigma_{s}}E_{ss}, -\sqrt{\sigma_{1}}E_{11},\ldots, -\sqrt{\sigma_{s}}E_{ss}\right),$$
and $V$ is an $(M+N)\times 2s$ matrix such that
$$V_{ab}= \begin{cases}
\delta_{a,i}/\sqrt{2} + \delta_{a,(M+i)}/\sqrt{2}, \ &b = i , \, i\le s,\\
\delta_{a,i}/\sqrt{2} - \delta_{a,(M+i)}/\sqrt{2}, \ &b = i+s , \, i\le s,\\
0, \ & b\ge 2s+1.
\end{cases}$$
With the identity
$$\det\left( {\begin{array}{*{20}c}
   { - I_{M\times M}} & DX  \\
   {(DX)^*} & { - zI_{N\times N}}  \\
\end{array}} \right) =\det(- I_{M\times M})\det(X^*D^*D X - zI_{N\times N}), $$
and Lemma 6.1 of \cite{KY}, if $\mu\notin \sigma((DX^s)^* DX^s)$, then $\mu$ is an eigenvalue of $\mathcal Q^\gamma := (X^s+\gamma E)^* D^* D(X^s+\gamma E)$ if and only if
\begin{equation}
\det(V^* G^s (\mu)V + (\gamma P_D)^{-1})=0,
\end{equation}
where
$$G^s(\mu):=\left(H^s-\left( {\begin{array}{*{20}c}
   {  I_{M\times M}} & 0  \\
   {0} & { \mu I_{N\times N}}  \\
\end{array}} \right)\right)^{-1}.$$ 
Define $Q^\gamma:=V^* G^s V + (\gamma P_D)^{-1}$ for $0<\gamma < 1$, it has the following $2\times 2$ blocks (recall the definition (\ref{Aij_group})):
\begin{equation}\label{Q_group}
\left( {\begin{array}{*{20}c}
   {Q^\gamma_{i,j}} & {Q^\gamma_{i,j+s}}  \\
   {Q^\gamma_{i+s,j}} & {Q^\gamma_{i+s,j+s}}  \\
   \end{array}} \right) = \frac{1}{2} \left( {\begin{array}{*{20}c}
   {1} & {1}  \\
   {1} & {-1}  \\
   \end{array}} \right) G_{[ij]} \left( {\begin{array}{*{20}c}
   {1} & {1}  \\
   {1} & {-1}  \\
   \end{array}} \right) + \delta_{ij}\left( {\begin{array}{*{20}c}
   {(\gamma E_{ii})^{-1}} & {0}  \\
   {0} & {-(\gamma E_{ii})^{-1}}  \\
   \end{array}} \right),  \ \ 1\le i \le s .
\end{equation}
Now let $\mu: = \lambda_1^s \pm N^{-3/4}.$ We claim that for all $0< \gamma \le 1$,
\begin{equation}\label{suff_claim}
\mathbb P\left(\det Q^\gamma (\mu) \ne 0\right) = 1-o(1).
\end{equation}
If (\ref{suff_claim}) holds, then $\mu$ is not an eigenvalue of $\mathcal Q^\gamma $ with probability $1-o(1)$. Denoting the largest eigenvalue of $\mathcal Q^\gamma $ by $\lambda^\gamma_1$ for $0<\gamma \leq 1$ and  $\lambda^{0}_1:=\lim_{\gamma\to 0}\lambda_1^\gamma$, hence we have $\lambda^0_1= \lambda_1^s $ and $\lambda^1_1= \lambda_1^E$ by definition. With the continuity of $\lambda_1^\gamma$ with respect to $\gamma$, the fact that $\lambda^0_1\in (\lambda_1^s -N^{-3/4},\lambda_1^s +N^{-3/4})$ and the eigenvalues are separated in the scale $N^{-2/3}$, we find that
$$ \lambda_1^E = \lambda^1_1 \in (\lambda_1^s-N^{-3/4},\lambda_1^s+N^{-3/4}),$$
with probability $1-o(1)$, i.e. we have proved (\ref{eq_claim}).

Finally, we prove the claim (\ref{suff_claim}). Choose $z=\lambda_r+iN^{-{2}/{3}}$ and note that $H^s$ has support bounded by $N^{-\epsilon}$. Then by (\ref{DIAGONAL}) and (\ref{SQUAREROOTBEHAVIOR}), we have with high probability
\begin{equation}
\max_{a} \vert G_{aa}^s (z) - \Pi_{aa}(\lambda_r) \vert \leq N^{-{\epsilon}/{2}}. \label{BOUND1}
\end{equation}
For the off-diagonal terms, we use (\ref{weak_off}), (\ref{SQUAREROOTBEHAVIOR}) and the Markov inequality to conclude that
\begin{equation}\label{BOUND2}
\max_{1\le a\ne b \le 2s} |G_{ab}^s(z)| \le N^{-1/6},
\end{equation}
holds with probability $1- o(N^{-1/6})$.
As pointed out in Remark \ref{rigid_multi}, we can extend (\ref{univ_small}) to finite correlation functions of largest eigenvalues. Since the largest eigenvalues in the Gaussian case are separated in the scale $\sim N^{-2/3}$, we conclude that
\begin{equation}\label{repulsion_estimate}
\mathbb P\left( \min_{i}|\lambda_i((X^s)^* X^s ) - \mu| \ge N^{-3/4} \right) \ge 1-o(1).
\end{equation}
On the other hand, the rigidity result (\ref{P1P1P1P1P1}) gives that with high probability,
\begin{equation}\label{rigid_estimate}
|\mu - \lambda_r | \le \varphi^{C_\xi}N^{-2/3}.
\end{equation}
Using (\ref{delocal}), (\ref{repulsion_estimate}), (\ref{rigid_estimate}) and the rigidity estimate (\ref{P1P1P1P1P1}), we can get with probability $1-o(1)$ that
\begin{equation}
\max_{a,b } \vert G_{ab}^{s}(z)-G_{ab}^{s}(\mu) \vert < N^{-{1}/{4}+\epsilon} . \label{REALCOMPLEX}
\end{equation}
For instance, for $\alpha, \beta \in \mathcal I_2$, small $c >0$ and large enough $C>0$, we have with probability $1-o(1)$ that
\begin{align*}
\left| G_{\alpha \beta}(z) - G_{\alpha \beta}(\mu)\right| & \le \sum_{k} \left|\zeta_k (\alpha) \zeta_k^*(\beta)\right|\left| \frac{1}{\lambda_k - z} - \frac{1}{\lambda_k - \mu}\right| \le \frac{\varphi^{C_1}}{N^{5/3}}\sum_{k} \frac{1}{|\lambda_k - z||\lambda_k - \mu|}  \\
& \le \frac{C\varphi^{C_1}}{N^{2/3}} + \frac{\varphi^{C_1}}{N^{5/3}}\sum_{1\le k \le \varphi^C} \frac{1}{|\lambda_k - z||\lambda_k - \mu|} + \frac{\varphi^{C_1}}{N^{5/3}}\sum_{k> \varphi^C, \lambda_k>\lambda_r - c} \frac{1}{|\lambda_k - z||\lambda_k - \mu|} \\
& \le \frac{\varphi^{C}}{N^{2/3}} + \frac{\varphi^{C_1+C}}{N^{1/4}} + \frac{\varphi^{C_1}}{N^{2/3}}\left(\frac{1}{N}\sum_{k> \varphi^C, \lambda_k>\lambda_r - c} \frac{1}{|\lambda_k - z||\lambda_k - \mu|}\right) \le N^{-1/4+\epsilon},
\end{align*}
where in the first step we used (\ref{spectral1}), in the second step (\ref{delocal}), in the third step $|\lambda_k - z||\lambda_k - \mu| \ge c^2$ for $\lambda_k \le \lambda_r-c$, in the fourth step (\ref{repulsion_estimate}), and in the last step the rigidity estimate (\ref{P1P1P1P1P1}). For all other choices of $a$ and $b$, we can prove the estimate (\ref{REALCOMPLEX}) in a similar way. Now by (\ref{REALCOMPLEX}), we see that (\ref{BOUND1}) and (\ref{BOUND2}) still hold if we replace $z$ by $\mu=\lambda_1^s \pm N^{-{3}/{4}}$ and double the right hand sides. Then using $\max_{i}|E_{ii}| \le \omega$ and (\ref{Q_group}), we get that for any $0<\gamma\le 1$,
\begin{align}
& \min_{1\le i \le s, \gamma} \{|Q^\gamma_{ii}|, |Q^\gamma_{i+s,i+s}|\} \ge {\omega}^{-1} - \frac{1}{2}\left| \Pi_{ii}(\mu)+ m_{2c}(\mu)\right| -  O(N^{-\epsilon/2}), \\
& \max_{1\le i \le s, \gamma} \{|Q^\gamma_{i,i+s}|, |Q^\gamma_{i+s,i}|\} \le  \frac{1}{2}\left| \Pi_{ii}(\mu) - m_{2c}(\lambda_r)\right| +  O(N^{-\epsilon/2}),
\end{align}
and
$$\max_{1 \le i \ne j \le s, \gamma} \left( |Q^\gamma_{i,j}|+|Q^\gamma_{i+s,j}|+|Q^\gamma_{i,j+s}|+ |Q^\gamma_{i+s,j+s}|\right) \le   O(N^{-1/6}),$$
hold with probability $1-o(1)$. Thus $Q^{\gamma}$ is diagonally dominant with probability $1-o(1)$, which proves the claim (\ref{suff_claim}).
\end{proof}

\section{Proof of Theorem \ref{thm_largerigidity} and Theorem \ref{lem_comparison}} \label{tools}


With Lemma \ref{lem_decrease}, given $X$ satisfying the assumptions in Theorem \ref{thm_largerigidity}, we can construct a matrix $\tilde{X}$ with support bounded by $q= O(N^{-1/2}\log N)$ and the same first four moments as $X$. Furthermore, $\tilde X$ satisfies the desired edge universality according to Lemma \ref{lem_smallcomp}. Then Theorem \ref{lem_comparison} will follow from the next lemma, which compares $X$ with $\tilde X$. 

\begin{lem}
Let $X$ and $\tilde{X}$ be two matrices as in Lemma \ref{lem_decrease}. Then there exist constants $\epsilon,\delta >0$ such that, for any $s \in \mathbb R$ we have
\begin{equation}\label{eq_edgeuniv}
\mathbb{P}^{\tilde X} (N^{\frac{2}{3}}(\lambda_1-\lambda_r)\leq s-N^{-\epsilon})-N^{-\delta} \leq \mathbb{P}^{X}(N^{\frac{2}{3}}({\lambda}_1-\lambda_r)\leq s)\leq \mathbb{P}^{\tilde X} (N^{\frac{2}{3}}(\lambda_1-\lambda_r)\leq s+ N^{-\epsilon})+N^{-\delta},
\end{equation}
where $\mathbb{P}^X$ and $\mathbb{P}^{\tilde X}$ are the laws for $X$ and $\tilde{X}$, respectively.
\end{lem}

%

By the rigidity result (\ref{P1P1P1P1P1}), we may assume that the parameter $s$ satisfies
\begin{equation}\label{PS}
|s| \le \varphi^{C_\xi},
\end{equation}
since otherwise (\ref{P1P1P1P1P1}) already gives the desired result.

Our goal is to write the distribution of the largest eigenvalue in terms of a cutoff function depending only on the Green functions. Then it is natural to use the Green function comparison method to conclude the proof. 
Let
\begin{equation*}
\mathcal{N}(E_1,E_2):= \#\{j: E_1 \leq \lambda_j \leq E_2\},
\end{equation*}
denote the number of eigenvalues of $\mathcal Q_2=X^{*}D^* DX$ in $[E_1, E_2]$; similarly we define $\tilde{\mathcal{N}}$ for $\tilde{\mathcal Q}_2=\tilde{X}^{*}D^* D\tilde{X} $. Hence to quantify the distribution of $\lambda_1$, it is equivalent to use $\mathbb{P}(\mathcal{N}(E, \infty)=0)$. 
Set
\begin{equation}\label{EU}
E_u:=\lambda_{r}+2 N^{-{2}/{3}}\varphi^{C_\xi},
\end{equation}
and for any $E\le E_u$ define $\mathcal{X}_E:=\mathbf{1}_{[E,E_u]}$ to be the characteristic function of the interval $[E,E_u]$. For any $\eta>0$, we define
$$\theta_{\eta}(x):=\frac{1}{\pi}\frac{\eta}{x^2+\eta^2}=\frac{1}{\pi}\Im \frac{1}{x-i\eta},$$
to be an approximate delta function on scale $\eta$. Note that under the above definitions, we have $\mathcal{N}(E,E_u)= \mathrm{Tr} \mathcal{X}_E(\mathcal Q_2)$ and
 \begin{equation}
 \mathrm{Tr} \mathcal{X}_{E-l} * \theta_{\eta}(\mathcal Q_2)= N \frac{1}{\pi} \int_{E-l}^{E_u} \operatorname{Im} m_2(y+i \eta)dy, \label{512}
 \end{equation}
for any $l>0$. 
Let $q$ be a smooth symmetric cutoff function such that
 \[q(x) = \begin{cases}    1 & \vert  x \vert \leq {1}/{9}\\
    0 & \vert x \vert \geq {2}/{9} \\
  \end{cases},\]
and we assume that $q(x)$ is decreasing when $x\geq 0$. Then the following lemma provides a way to approximate $\mathbb{P}(\mathcal{N}(E, \infty)=0) $ with a function depending only on Green functions. 

\begin{lem} \label{lem_cutoff2}
For $\epsilon>0$, let $\eta=N^{-{2}/{3}-9\epsilon}$ and $l=N^{-{2}/{3}-\epsilon}/2$. Suppose Theorem \ref{thm_largerigidity} holds.
Then for all $E$ such that
\begin{equation*}
\vert E-\lambda_r \vert \leq \frac{3}{2} \varphi^{C_\xi} N^{-{2}/{3}},
\end{equation*}
where the constant $C_\xi$ is as in (\ref{P1P1P1P1P1}), (\ref{PPPPPP}), (\ref{PS}) and (\ref{EU}), we have
\begin{equation}\label{eq_cutoff}
\mathbb{E}q(\mathrm{Tr}\mathcal{X}_{E-l}*\theta_{\eta}(\mathcal Q_2)) \leq \mathbb{P}(\mathcal{N}(E, \infty)=0) \leq \mathbb{E}q(\mathrm{Tr}\mathcal{X}_{E+l}*\theta_{\eta}(\mathcal Q_2))+ \exp(-c\varphi^{C_\xi}),
\end{equation}
for some constant $c>0$.
\end{lem}
\begin{proof}
See \cite[Corollary 4.2]{PY} or \cite[Corollary 6.2]{EYY} for the proof. The key inputs are the rigidity estimates (\ref{P1P1P1P1P1}) and (\ref{PPPPPP}) in Theorem \ref{thm_largerigidity}.
\end{proof}
To prove Lemma \ref{eq_edgeuniv}, we need the following Green function comparison result, which will be proved in Section \ref{comparison}.

\begin{lem} \label{lem_compdiffsupport} 
Let $X$ and $\tilde{X}$ be two matrices as in Lemma \ref{lem_decrease}. Suppose $F:\mathbb R\to \mathbb R$ is a function whose derivatives satisfy
\begin{equation}
\sup_{x} \vert {F^{n}(x)}\vert {(1+\vert x \vert)^{-C_3}} \leq C_3, \ n=1,2,3, \label{FCondtion}
\end{equation}
for some constant $C_3>0$. Then for any sufficiently small $\epsilon>0 $ and for any real numbers
\begin{equation*}
E,E_1,E_2 \in I_{\epsilon}:=\left\{x: \vert x -\lambda_{r}\vert \leq N^{-{2}/{3}+\epsilon}\right\} \ \ \text{and} \ \ \eta:=N^{-{2}/{3}-\epsilon},
\end{equation*}
we have
\begin{equation}
\left| \mathbb{E}F\left(N\eta \operatorname{Im} m_2(z)\right)-\mathbb{E}F\left(N\eta \operatorname{Im} \tilde{m}_2 (z)\right) \right| \leq N^{-\phi + C_4 \epsilon}, \ \ z=E+i\eta,  \label{BDBD}
\end{equation}
and
\begin{equation}
\left| \mathbb{E}F\left(N \int_{E_1}^{E_2} \operatorname{Im} m_2(y+i\eta)dy\right)- \mathbb{E} F\left(N \int_{E_1}^{E_2} \operatorname{Im} \tilde{m}_2(z)(y+i\eta)dy\right) \right| \leq N^{-\phi+ C_4\epsilon}, \label{BDBD1}
\end{equation}
where $\phi $ is defined in Theorem \ref{thm_largerigidity} and $ C_4>0$ is some constant. 
\end{lem}

\begin{proof}[Proof of Lemma \ref{eq_edgeuniv}]
Recall that we only consider $s$ satisfying (\ref{PS}), so it suffices to assume $|E-\lambda_r|\le N^{-2/3}\varphi^{C_\xi}$. Then by Lemma \ref{lem_compdiffsupport} and (\ref{512}), there exists $\delta>0$ such that
\begin{equation}
 \mathbb{E}q\left(\mathrm{Tr}\mathcal{X}_{E-l}*\theta_{\eta}\left(\tilde{\mathcal Q}_2\right) \right) \leq \mathbb{E}q\left(\mathrm{Tr}\mathcal{X}_{E-l}*\theta_{\eta}(\mathcal Q_2)\right)+N^{-\delta}. \label{1}
\end{equation}
For the choice $l=\frac{1}{2} N^{-{2}/{3}-\epsilon}$, we also have $|E-l - \lambda_r| \le \frac{3}{2} \varphi^{C_\xi} N^{-{2}/{3}}$. Thus we can apply Lemma \ref{lem_cutoff2} to get
\begin{equation}
\mathbb{P}\left(\tilde{\mathcal{N}}(E-2l,\infty)=0\right) \leq \mathbb{E}q\left(\mathrm{Tr}\mathcal{X}_{E-l}*\theta_{\eta}\left(\tilde{\mathcal Q}_2\right)\right)+\exp(-c \varphi^{C_\xi}) .\label{2}
\end{equation}
With (\ref{1}), (\ref{2}) and Lemma \ref{lem_cutoff2}, we get that
\begin{equation}
\mathbb{P}(\tilde{\mathcal{N}}(E-2l,\infty)=0)-2N^{-\delta}\le  \mathbb{E}q\left(\mathrm{Tr}\mathcal{X}_{E-l}*\theta_{\eta}(\mathcal Q_2)\right) \leq \mathbb{P}(\mathcal{N}(E,\infty)=0) .\label{ALMOSTLAST}
\end{equation}
If we choose $E=\lambda_r+sN^{-\frac{2}{3}}$, then (\ref{ALMOSTLAST}) implies that
\begin{equation*}
 \mathbb{P}^{\tilde X}\left(N^{{2}/{3}}\left({\lambda}_1-\lambda_{r}\right)\leq s-N^{-\epsilon}\right)-2N^{-\delta} \leq \mathbb{P}^X\left(N^{{2}/{3}}\left(\lambda_1-\lambda_{r}\right)\leq s\right).
 \end{equation*}
This proves one inequality in (\ref{eq_edgeuniv}). The other inequality can be proved in a similar way using Lemma \ref{lem_cutoff2} and Lemma \ref{lem_compdiffsupport}.
\end{proof}


Finally, we prove Theorem \ref{thm_largerigidity} with the following lemma. Its proof will be given in Section \ref{comparison}.
\begin{lem} \label{lem_comgreenfunction}
Let $X$ and $\tilde{X}$ be two matrices as in Lemma \ref{lem_decrease}. For $z \in S(c_1, C_0,C_1)$ with large enough $C_1$, if there exist deterministic quantities $J\equiv J(N)$ and $K\equiv K(N)$, such that
\begin{equation}
\max_{a \neq b} \vert \tilde{G}_{ab}(z) \vert \leq J, \ \ \ \vert \tilde m_2(z)-m_{2c}(z) \vert \leq K, \label{KEYBOUNDS}
\end{equation}
hold with $\xi_1$-high probability for some $\xi_1\ge 3$, then for any $p \in 2\mathbb{N}$ with $p \leq \varphi$, we have
\begin{equation}
\mathbb{E} \vert m_2(z)-m_{2c}(z) \vert ^p \leq \mathbb{E} \vert \tilde m_2(z) - m_{2c}(z) \vert ^p+ (Cp)^{Cp}(J^2+K + N^{-1})^p. \label{KEYEYEYEY}
\end{equation}
\end{lem}
\begin{proof}[Proof of Theorem \ref{thm_largerigidity}]
By Lemma \ref{lem_decrease}, $\tilde{X}$ has support bounded by $q= O(N^{-1/2}\log N)$. Then using (\ref{SQUAREROOTBEHAVIOR}), we can get that
$$q =  O\left( {\log N}\sqrt{\frac{\Im\, m_{2c}}{N\eta}}\right), \ \ \frac{q^2}{\sqrt{\kappa+\eta}} = O\left(\frac{(\log N)^2}{N\eta}\right).$$
Thus (\ref{MPBOUNDS}) and (\ref{DIAGONAL}) show that we can choose
\begin{equation*}
J= \varphi^{C_5/2}\left(\frac{1}{N\eta}+\sqrt{\frac{\operatorname{Im} m_{2c}}{N \eta}}\right) \ \ \text{and} \ \ K=\frac{\varphi^{C_5}}{N \eta}, 
\end{equation*}
for some large enough $C_5>0$ such that (\ref{KEYBOUNDS}) holds with $\xi_1$-high probability. Then using Markov inequality and (\ref{KEYEYEYEY}), we get that for sufficiently large $C_2>0$ and sufficiently small $c,c'>0$,
\begin{align}
\mathbb{P}\left(\vert m_2(z)-m_{2c}(z) \vert \geq \varphi^{C_2}(N \eta)^{-1}\right) \leq \frac{\left(C\eta^{-1}\right)^p \exp(-c\varphi^{\xi_1})}{\varphi^{C_2 p}(N \eta)^{-p}} + \frac{(Cp)^{Cp}\varphi^{C_5p} }{\varphi^{C_2p} }  \leq \exp(-c'\varphi), \label{STRINGLING}
\end{align}
where we used $p=\varphi$ and the trivial bound $\left|\tilde m_2(z) -  m_{2c}(z) \right|\le C\eta^{-1} $ (see (\ref{eq_gbound})) on the bad event with probability $\le \exp(-c\varphi^{\xi_1})$. This proves (\ref{NEWMPBOUNDS}). Then using (\ref{NEWMPBOUNDS}), we can derive the rigidity results (\ref{P1P1P1P1P1}) and (\ref{PPPPPP}) with the arguments in Section 8 of \cite{EKYY1}, Section 5 of \cite{EYY} or Lemma 8.1 of \cite{PY}. 
\end{proof}

\section{Proof of Lemma \ref{lem_compdiffsupport}, Lemma \ref{lem_comgreenfunction} and Lemma \ref{thm_largebound}}  \label{comparison}

To prove Lemma \ref{lem_compdiffsupport}, Lemma \ref{lem_comgreenfunction} and Lemma \ref{thm_largebound}, we will use the Green function comparison method developed in \cite{LY}. More specifically, we will apply the Lindeberg replacement strategy to $G$ in (\ref{eqn_defG}). Let $X=(x_{i\mu})$ and $\tilde{X}=(\tilde x_{i\mu})$ be two matrices as in Lemma \ref{lem_decrease}. Define a bijective ordering map $\Phi$ on the index set of $X$ as
\begin{equation*}
\Phi: \{(i,\mu):1 \leq i \leq M, \ M+1 \leq \mu \leq M+N \} \rightarrow \{1,\ldots,\gamma_{\max}=MN\}.
\end{equation*} 
For any $1\le \gamma \le \gamma_{\max}$, we define the matrix $X^{\gamma}= \left(x^{\gamma}_{i\mu}\right)$ such that $x_{i\mu}^{\gamma} =X_{i\mu} $ if $\Phi(i,\mu)\leq \gamma$, and $x_{i\mu}^{\gamma} =\tilde{X}_{i\mu}$ otherwise. Note that we have $X^0=\tilde X$, $X^{\gamma_{\max}}=X$, and $X^\gamma$ satisfies the bounded support condition with $q=N^{-\phi}$ for all $0\le \gamma \le \gamma_{\max}$. Correspondingly, we define
 \begin{equation}\label{Hgamma}
   H^{\gamma} := \left( {\begin{array}{*{20}c}
   { 0 } & DX^\gamma  \\
   {(DX^\gamma)^*} & {0}  \\
   \end{array}} \right), \ \ \ G^\gamma:= \left( {\begin{array}{*{20}c}
   { - I_{M\times M}} & DX^\gamma  \\
   {(DX^\gamma)^*} & { - zI_{N\times N}}  \\
\end{array}} \right)^{-1}.
 \end{equation}
Note that $H^{\gamma}$ and $H^{\gamma-1}$ differ only at $(i,\mu)$ and $(\mu,i)$ elements, where $\Phi(i,\mu) = \gamma$. Then we define the $(N+M)\times (N+M)$ matrices $V$ and $W$ by
$$V_{ab}=\left(\mathbf{1}_{\{(a,b)=(i,\mu)\}} + \mathbf{1}_{\{(a,b)=(\mu,i)\}}\right) \sqrt{\sigma_i}x_{i\mu}, \ \ W_{ab}=\left(\mathbf{1}_{\{(a,b)=(i,\mu)\}}+ \mathbf{1}_{\{(a,b)=(\mu,i)\}}\right) \sqrt{\sigma_i}\tilde x_{i\mu},$$
so that $H^{\gamma}$ and $H^{\gamma-1}$ can be written as
$$H^\gamma= Q + V, \ \ H^{\gamma-1} = Q+W,$$
for some $(N+M)\times (N+M)$ matrix $Q$ satisfying $Q_{i\mu}=Q_{\mu i}=0$. For simplicity of notations, we denote the Green functions
\begin{equation}
S:=G^\gamma, \  \  T:=G^{\gamma-1}, \ \ R:=\left(Q-\left( {\begin{array}{*{20}c}
   {  I_{M\times M}} & 0  \\
   {0} & { z I_{N\times N}}  \\
\end{array}} \right)\right)^{-1}.  \label{R}
\end{equation}
Under the above definitions, we can write
\begin{align}
S= \left(Q -\left( {\begin{array}{*{20}c}
   {  I_{M\times M}} & 0  \\
   {0} & { z I_{N\times N}}  \\
\end{array}} \right) + V\right)^{-1}=(1+RV)^{-1}R.  \label{RESOLVENT}
\end{align}
Thus we can expand $S$ using the resolvent expansion till order $m$:
\begin{equation}
S=R-RVR+(RV)^2R+\ldots+(-1)^m(RV)^m R+(-1)^{m+1}(RV)^{m+1}S. \label{RESOLVENTEXPANSION}
\end{equation}
On the other hand, we can also expand $R$ in terms of $S$,
\begin{equation}
R=(1-SV)^{-1}S=S+SVS+(SV)^2 S+\ldots+ (SV)^m S+ (SV)^{m+1}R. \label{RESOLVENTEXPANSION2}
\end{equation}
We have similar expansions for $T$ and $R$ by replacing $V$, $S$ with $W$, $T$ in (\ref{RESOLVENTEXPANSION}) and (\ref{RESOLVENTEXPANSION2}).  

With the bounded support condition
\begin{equation}
\max_{a,b} \vert V_{ab} \vert = \sqrt{\sigma_i}|x_{i\mu}| = O(N^{-\phi}) \label{5BOUND2},
\end{equation}
with $\xi_1$-high probability. Together with Lemma \ref{lem_small}  and (\ref{RESOLVENTEXPANSION2}), it is easy to check that
$\max_{a,b}|R_{ab}| =O(1)$ with $\xi_1$-high probability. 
Hence by (\ref{DIAGONAL}), there exists a constant $C_6>0$ such that with $\xi_1$-high probability, 
\begin{equation}
\sup_{z\in S(c_1,C_0,C_1)} \max_{\gamma} \max_{a,b} \max \left\{ \vert S_{ab} \vert, |T_{ab}|, |R_{ab}| \right\}\leq C_6, \label{5BOUND1}
\end{equation}
where we uses that $|m_{2c}(z)| $ is uniformly bounded on $S(c_1,C_0,C_1)$. On the other hand, we have the following trivial deterministic bound for $S$, $R$ and $T$ (see (\ref{eq_gbound})):
\begin{equation}
\sup_{z\in S(C_0)} \max_{\gamma} \max_{a,b} \max \left\{ \vert S_{ab} \vert, |T_{ab}|, |R_{ab}| \right\}\le  C\eta^{-1} \le CN. \label{5BOUNDT}
\end{equation}
In the following discussions, we fix $\gamma$ and $i,\mu$ such that $\Phi(i,\mu)=\gamma$. The expressions below will depend on $\gamma$, but we drop the subscripts for convenience. For simplicity, we will use $|\mathbf v| \equiv \|\mathbf v\|_1$ to denote the $l^1$-norm for any vector $\bf v$. 

The following lemma gives a simple estimate for the remainder terms in (\ref{RESOLVENTEXPANSION}) and (\ref{RESOLVENTEXPANSION2}). 

\begin{lem} \label{resolvent_error} 
For the resolvent expansion (\ref{RESOLVENTEXPANSION}) and (\ref{RESOLVENTEXPANSION2}), we have
\begin{equation}
\max_{a,b}\max\left\{\left|((RV)^{m}S)_{ab}\right| , \left|((SV)^{m}R)_{ab}\right|\right\} =O(N^{-m \phi}). \label{REMINDERBOUND}
\end{equation}
for any fixed $m\in \mathbb N$.
\end{lem} 
\begin{proof}
By the definition of $V$, for example, we have 
\begin{equation}\label{graphical}
([RV]^m S)_{ab}= \sum_{(a_i,b_i) \in \{(i,\mu), (\mu,i)\}:1 \leq i \leq m} (\sqrt{\sigma_i}x_{i\mu})^m R_{a a_1} R_{b_1 a_2} \cdots S_{b_m b} .
\end{equation}
Since there are only finitely many terms in the above sum for fixed $m$, the conclusion follows immediately from (\ref{5BOUND2}) and (\ref{5BOUND1}).
\end{proof}

From the expression (\ref{graphical}), one can see that it is helpful to introduce the following notations.

\begin{defn}[Matrix operators $*_\gamma$] \label{def_operator1}
For any two $(N+M) \times (N+M)$ matrices $A$ and $B$, we define $A *_{\gamma} B$ as
\begin{equation}
A *_{\gamma} B=A I_{\gamma} B, \ \ \ (I_{\gamma})_{ab}=\mathbf{1}_{\{(a,b)=(i,\mu)\}}+\mathbf{1}_{\{(a,b)=(\mu,i)\}},
\end{equation} 
where $(i,\mu)$ is such that $\Phi(i,\mu)=\gamma$. In other words, we have
$$(A*_{\gamma}B)_{ab}=A_{ai}B_{\mu b}+A_{a\mu}B_{ib}.$$ 
When $\gamma$ is fixed, we often drop the subscript $\gamma$ and write $A*B$ for simplicity. Also we denote the $m$-th power of $A$ under the $*_\gamma$-product by $A^{*m}$, i.e.
\begin{equation}
A^{* m}\equiv A^{*_\gamma m}:=A*A*A*\ldots*A.
\end{equation}

\end{defn}


\begin{defn} [$\mathcal{P}_{\gamma, \mathbf{k}}$ and $\mathcal{P}_{\gamma,k}$ notations] \label{def_operator2}
 For $k \in \mathbb{N}$ and $\mathbf{k}=(k_1, \cdots, k_s) \in \mathbb{N}^s$, $ \gamma=\Phi(i,\mu)$, we define
 \begin{equation}
 \mathcal{P}_{\gamma, k} G_{ab} = G_{ab}^{*_{\gamma}(k+1)},
 \end{equation}
and 
 \begin{equation} 
 \mathcal{P}_{\gamma,\mathbf{k}}\left(\prod_{t=1}^s G_{a_t b_t}\right):=\prod_{t=1}^s \left(\mathcal{P}_{\gamma, k_t} G_{a_t b_t}\right)=\prod_{t=1}^s G_{a_t b_t}^{*_{\gamma}(k_t+1)}.
 \end{equation}
If $G_1$ and $G_2$ are products of matrix entries as above, then we define
\begin{equation}
 \mathcal{P}_{\gamma,\mathbf{k}} (G_1+G_2):= \mathcal{P}_{\gamma,\mathbf{k}}G_1 +  \mathcal{P}_{\gamma,\mathbf{k}}G_2.
\end{equation}
Similarly, for the product of the entries of $G-\Pi$, we define
 \begin{equation}
 \tilde{\mathcal{P}}_{\gamma, \mathbf{k}}\left(\prod_{t=1}^s (G-\Pi)_{a_tb_t}\right)=\prod_{t=1}^s\left(\tilde{\mathcal{P}}_{\gamma, k_t}(G-\Pi)_{a_tb_t}\right),
 \end{equation}
 where 
 \[ \tilde{\mathcal{P}}_{\gamma, k}(G-\Pi)_{a b}=\begin{cases} 
      (G-\Pi)_{ab} & \text{if } a=b \ \text{and} \ k_t=0, \\
       G_{a_tb_t}^{*(k_t+1)} &  \text{otherwise}.
   \end{cases}\]
Again, we will often drop the subscript $\gamma$ whenever there is no confusion about $\gamma$.
\end{defn}

\begin{rem}
Note that $ \mathcal{P}_{\gamma,\mathbf{k}}$ and $\tilde P_{\gamma,\mathbf k}$ are not linear operators acting on matrices, but just notations we use for simplification. Moreover, for $k, l \in \mathbb{N}$ and $\mathbf{k} \in \mathbb{N}^{k+1}$, it is easy to verify that
\begin{equation}
G^{*l} I_{\gamma} G^{*k}=G^{*(l+k)}, \ \ \ \mathcal{P}_{\gamma, \mathbf{k}}(\mathcal{P}_{{\gamma},k}G_{ab})=\mathcal{P}_{\gamma, k+ \vert \mathbf{k} \vert} G_{ab} .\label{FIRST}
\end{equation}
For the second equality, note that $\mathcal{P}_{{\gamma},k}G_{ab}$ is a sum of products of the entries of $G$, where each product contains $k+1$ matrix entries.
\end{rem}

With the above definitions and bound (\ref{5BOUND1}), it is easy to prove the following lemma. 

\begin{lem} \label{propsotion_poperator} 
 For any $\mathbf{k} \in \mathbb{N}^s, \ \gamma, $ and $a_1,b_1,\cdots, a_s,b_s$,  with $\xi_1$-high probability we have that
\begin{equation}
\max\left\{ \left\vert \mathcal{P}_{\gamma, \mathbf{k}}\left(\prod_{t=1}^s A_{a_t b_t}\right) \right\vert, \ \ \left\vert \tilde{P}_{\gamma, \mathbf{k}}\left(\prod_{t=1}^s (A-\Pi)_{a_tb_t}\right) \right\vert \right\} \leq  (2C_6)^{s+ \vert \mathbf{k} \vert+1}, \label{SECOND}
\end{equation}
where $A$ can be $R$, $S$, or $T$.
\end{lem}


Now we begin to perform the Green function comparison strategy. The basic idea is to expand $S$ and $T$ in terms of $R$ using the resolvent expansions (\ref{RESOLVENTEXPANSION}) and (\ref{RESOLVENTEXPANSION2}), and then compare the two expressions. We expect that the main terms will cancel since $x_{i\mu}$ and $\tilde x_{i\mu}$ have the same first four moments, while the remaining error terms will be sufficiently small since $x_{i\mu}$ and $\tilde x_{i\mu}$ have support bounded by $N^{-\phi}$. The key is the following Lemma \ref{Greenfunctionrepresent}, whose proof can be found in \cite[Section 6]{LY}. 


\begin{lem} [Green function representation theorem] \label{Greenfunctionrepresent} 
Let $z\in S(c_1,C_0,C_1)$ as in (\ref{5BOUND1}) and $\Phi(i,\mu)=\gamma$. Fix $s=O(\varphi)$ and $\zeta=O(\varphi)$. Then for any $S_{a_1 b_1}(z)S_{a_2 b_2}(z)\cdots S_{a_sb_s}(z)$, we have 
\begin{equation}
\mathbb{E}\prod_{t=1}^s S_{a_tb_t} = \sum_{0 \leq k \leq 4} A_{k} \mathbb{E}\left[(-\sqrt{\sigma_i}x_{i\mu})^{k}\right]+\sum_{5 \leq \vert \mathbf{k} \vert \leq {2 \zeta}/{\phi}, \mathbf k\in \mathbb N^s } \mathcal{A}_{ \mathbf{k}  }\mathbb{E} \, \mathcal{P}_{\gamma,\mathbf{k} }\prod_{t=1}^s S_{a_t,b_t}+O(N^{-\zeta}), \label{ONLYPROVE}
\end{equation}
where $A_k$, $0\le k \le 4$, depend only on $R$, and $\mathcal A_{\mathbf k}$'s depend both on $R$ and $S$ but are independent of $(a_t, b_t)$, $1\leq t \leq s$. Moreover, we have the estimate 
\begin{equation}
\vert \mathcal{A}_{  \mathbf{k} } \vert \leq N^{-{\vert \mathbf{k} \vert  \phi} /{10}-2} . \label{INDEXBOUND}
\end{equation}
Similarly, we have
\begin{equation}
\mathbb{E}\prod_{t=1}^s (S-\Pi)_{a_tb_t} = \sum_{0 \leq k \leq 4} \tilde{A}_{k} \mathbb{E}\left[(-\sqrt{\sigma_i}x_{i\mu})^{k}\right]+\sum_{5 \leq \vert \mathbf{k} \vert \leq {2 \zeta}/{\phi}, \mathbf k \in \mathbb N^s} \mathcal{A}_{\mathbf{k}}\mathbb{E} \, \mathcal{\tilde{P}}_{\gamma,\mathbf{k}}\prod_{t=1}^s S_{a_t,b_t}+O(N^{-\zeta}),  \label{ONLYPROVE2}
\end{equation}
where $\tilde{A}_k$, $0\leq k \leq 4$, depend only on $R$, and $\mathcal{A}_{\mathbf{k}}$ are the same as above. 

Finally, as (\ref{ONLYPROVE}), we have
\begin{equation}
\mathbb{E}\prod_{t=1}^s S_{a_tb_t} = \mathbb{E}\prod_{t=1}^s R_{a_tb_t}+\sum_{1 \leq \vert  \mathbf{k}\vert \leq {2 \zeta}/{\phi}, \mathbf k \in \mathbb N^s} \tilde{\mathcal{A}}_{\mathbf{k}}\mathbb{E} \, \mathcal{P}_{\gamma,\mathbf{k}} \prod_{t=1}^s S_{i_t,j_t}+O(N^{-\zeta}), \label{ONLYPROVE3}
\end{equation}
where $\mathcal{\tilde{A}}$ are independent of $(a_t, b_t)$, $1\leq t \leq s$, and 
\begin{equation}
\vert \mathcal{\tilde{A}}_{\mathbf{k}} \vert \leq N^{-\vert \mathbf k \vert \phi /10} .
\end{equation}
Note that the terms $\mathcal A$ and $\tilde{\mathcal A}$ do depend on $\gamma$ and we have omitted this dependence in the above formulas.
\end{lem}

Now we use Lemma \ref{Greenfunctionrepresent} to complete the proof of Lemma \ref{thm_largebound}, Lemma \ref{lem_compdiffsupport} and Lemma \ref{lem_comgreenfunction}.

\begin{proof}[Proof of Lemma \ref{thm_largebound}]
It is clear that a result similar to Lemma \ref{Greenfunctionrepresent} also holds for the product of $T$ entries. Thus as in (\ref{ONLYPROVE}), we define the notation $\mathcal{A}^{\gamma,a}$, $a=0,1$ as follows:
\begin{equation} \label{greenpowers}
\mathbb{E}\prod_{t=1}^s S_{a_tb_t} = \sum_{0 \leq k \leq 4} A_{k} \mathbb{E}\left[(-\sqrt{\sigma_i}x_{i\mu})^{k}\right]+\sum_{5 \leq \vert \mathbf{k} \vert \leq {2 \zeta}/{\phi}, \mathbf k\in \mathbb N^s } \mathcal{A}_{  \mathbf{k}}^{\gamma, 0}\mathbb{E} \, \mathcal{P}_{\gamma,\mathbf{k} }\prod_{t=1}^s S_{a_t,b_t}+O(N^{-\zeta}),
\end{equation}
\begin{equation} \label{greenpowert}
\mathbb{E}\prod_{t=1}^s T_{a_tb_t} = \sum_{0 \leq k \leq 4} A_{k} \mathbb{E}\left[(-\sqrt{\sigma_i}\tilde{x}_{i\mu})^{k}\right]+\sum_{5 \leq \vert \mathbf{k} \vert \leq {2 \zeta}/{\phi}, \mathbf k\in \mathbb N^s } \mathcal{A}_{  \mathbf{k}}^{\gamma, 1}\mathbb{E} \, \mathcal{P}_{\gamma,\mathbf{k} }\prod_{t=1}^s T_{a_t,b_t}+O(N^{-\zeta}).
\end{equation}
Since $A_{k}$, $0 \leq k \leq 4$ depend only on $R$ and $x_{i \mu}$, $\tilde{x}_{i \mu}$ have the same first four moments, we get from (\ref{greenpowers}) and (\ref{greenpowert}) that for $s=O(\varphi)$ and $\zeta =O(\varphi)$,
\begin{align} 
& \mathbb{E}\prod_{t=1}^s G_{a_tb_t}-\mathbb{E}\prod_{t=1}^s \tilde{G}_{a_tb_t} =\sum_{\gamma=1}^{\gamma_{\max}} \left(\mathbb{E}\prod_{t=1}^s G^{\gamma}_{a_tb_t}-\mathbb{E}\prod_{t=1}^s G^{\gamma-1}_{a_tb_t}\right) \nonumber\\
& = \sum_{\gamma=1}^{\gamma_{\max}} \sum_{5 \leq \vert \mathbf{k} \vert \leq {2 \zeta}/{\phi},\mathbf k\in \mathbb N^s}\left(\mathcal{A}^{\gamma,0}_{  \mathbf{k}  }\mathbb{E} \,\mathcal{P}_{\gamma,\mathbf{k} }\prod_{t=1}^s G^{\gamma}_{a_t,b_t}-\mathcal{A}_{  \mathbf{k}}^{\gamma,1}\mathbb{E} \, \mathcal{P}_{\gamma,\mathbf{k}}\prod_{t=1}^s G^{\gamma-1}_{a_t,b_t}\right)+O(N^{-\zeta}). \label{telescoping_2}
\end{align}
Then we obtain that
\begin{equation} \label{teles_inequality}
\left\vert \mathbb{E}\prod_{t=1}^s G_{a_tb_t}^{\gamma_{\max}}  \right\vert \leq \left\vert \mathbb{E}\prod_{t=1}^s G_{a_tb_t}^{0} \right\vert+ \sum_{\gamma=1}^{\gamma_{\max}} \sum_{a=0,1} \sum_{5 \leq \vert \mathbf{k} \vert \leq {2 \zeta}/{\phi},\mathbf k\in \mathbb N^s} \vert \mathcal{A}_{\mathbf{k}}^{\gamma, a}  \vert \left\vert  \mathbb{E} \mathcal{P}_{\gamma,\mathbf{k}}\prod_{t=1}^s G^{\gamma-a}_{a_t,b_t} \right\vert+O(N^{-\zeta}).
\end{equation}
By (\ref{SECOND}) and  (\ref{INDEXBOUND}), the second term in (\ref{teles_inequality}) is bounded by 
\begin{align}
&\sum_{5 \leq k \leq {2 \zeta}/{\phi}} \sum_{\gamma=1}^{\gamma_{\max}} \sum_{a=0,1} \sum_{\vert \mathbf{k} \vert = k, \mathbf k \in \mathbb N^s} \vert \mathcal{A}_{\mathbf{k}}^{\gamma, a} \vert  \left\vert \mathbb{E} \mathcal{P}_{\gamma,\mathbf{k}}\prod_{t=1}^s G^{\gamma-a}_{a_t,b_t} \right\vert \nonumber\\
& \leq C \sum_{5 \leq k \leq {2 \zeta}/{\phi}} N^{-{k \phi}/{10}} s^k (2C_6)^{s+k+1} \le C N^{-{5 \phi}/{10}} s^5 (2C_6)^{s+5+1}  \le N^{-{5 \phi}/{20}} (2C_6)^s, \label{iterate0}
\end{align}
where we used the rough bound $\#\{\mathbf k \in \mathbb N^s: |\mathbf k|=s \} \le s^k$ and $s= O(\varphi)$. 

However, the bound in (\ref{iterate0}) is not good enough. To improve it, we iterate the above arguments as following. Recall that $\mathcal{P}_{\gamma,\mathbf{k}}\prod_{t=1}^s G^{\gamma-a}_{a_t,b_t}$ is also a sum of products of $G$. Applying (\ref{telescoping_2}) again to the term  $\mathbb{E} \mathcal{P}_{\gamma,\mathbf{k}}\prod_{t=1}^s G^{\gamma-a}_{a_t,b_t}$ and replacing $\gamma_{\max}$ in (\ref{teles_inequality}) with $\gamma-a$, we obtain that
\begin{equation}\label{teles_inequality1}
\left\vert \mathbb{E} \mathcal P_{\gamma,\mathbf k}\prod_{t=1}^s G_{a_tb_t}^{\gamma-a} \right \vert \leq \left\vert \mathbb{E} \mathcal P_{\gamma,\mathbf k} \prod_{t=1}^s G_{a_tb_t}^{0} \right \vert+ \sum_{\gamma^{\prime}=1}^{\gamma-a} \sum_{a^{\prime}=0,1} \sum_{5 \leq \vert \mathbf{k}^{\prime} \vert \leq {2 \zeta}/{\phi},\mathbf k^{\prime}\in \mathbb N^{s+\vert \mathbf{k} \vert}} \vert \mathcal{A}_{\mathbf{k}^{\prime}}^{\gamma^{\prime}, a^{\prime}} \vert  \vert \mathbb{E} \mathcal{P}_{\gamma^{\prime},\mathbf{k}^{\prime}}\mathcal{P}_{\gamma,\mathbf{k}}\prod_{t=1}^s G^{\gamma^{\prime}-a^{\prime}}_{a_tb_t} \vert+ O(N^{-\zeta}).
\end{equation}
Together with (\ref{teles_inequality}), we  have
\begin{align}
\left\vert \mathbb{E}\prod_{t=1}^s G^{\gamma_{\max}}_{a_tb_t} \right\vert \leq  \left\vert \mathbb{E}\prod_{t=1}^s G^{0}_{a_tb_t} \right\vert+ \sum_{\gamma=1}^{\gamma_{\max}} \sum_{a=0,1} \sum_{5 \leq \vert \mathbf{k} \vert \leq {2 \zeta}/{\phi},\mathbf k\in \mathbb Z^s_+} \left\vert \mathcal{A}_{\mathbf{k}}^{\gamma, a} \right\vert  \left\vert \mathbb{E} \mathcal{P}_{\gamma,\mathbf{k}}\prod_{t=1}^s G^{0}_{a_t,b_t}\right \vert    \nonumber \\
+ \sum_{\gamma, \gamma^{\prime}} \sum_{a, a^{\prime}} \sum_{\mathbf{k}, \mathbf{k}^{\prime}} \left| \mathcal{A}_{\mathbf{k}}^{\gamma,a} \mathcal{A}_{\mathbf{k}^{\prime}}^{\gamma^{\prime},a^{\prime}} \right| \left\vert \mathbb{E} \mathcal{P}_{\gamma^{\prime},\mathbf{k}^{\prime}}\mathcal{P}_{\gamma,\mathbf{k}}\prod_{t=1}^s G^{\gamma^{\prime}-a^{\prime}}_{a_tb_t} \right\vert+O(N^{-\zeta}).
\end{align}
Again using (\ref{SECOND}) and  (\ref{INDEXBOUND}), it is easy to obtain that
\begin{equation}
\sum_{\gamma, \gamma^{\prime}} \sum_{a, a^{\prime}} \sum_{\mathbf{k}, \mathbf{k}^{\prime}} \left\vert \mathcal{A}_{\mathbf{k}}^{\gamma,a} \mathcal{A}_{\mathbf{k}^{\prime}}^{\gamma^{\prime},a^{\prime}} \right\vert  \left| \mathbb{E} \mathcal{P}_{\gamma^{\prime},\mathbf{k}^{\prime}}\mathcal{P}_{\gamma,\mathbf{k}}\prod_{t=1}^s G^{\gamma^{\prime}-a^{\prime}}_{a_tb_t} \right| \leq N^{-\frac{10 \phi}{20}} (2C_6)^s, \label{process_end}
\end{equation}
where we used that $k' + k\ge 10$. Repeating the above process for $n \leq {6 \zeta}/{\phi}$ times, we obtain that
\begin{equation} \label{teles_inequality2}
\left\vert \mathbb{E}\prod_{t=1}^s G^{\gamma_{\max}}_{a_tb_t} \right\vert  \leq  \sum_{n=0}^{{6 \zeta}/{\phi}} \sum_{\gamma_1, \cdots, \gamma_n} \sum_{a_1,\cdots, a_n} \sum_{\mathbf{k}_1,\cdots, \mathbf{k}_n} \left\vert \prod_{j} \mathcal{A}_{\mathbf{k}_j}^{\gamma_j,a_j} \right\vert \left\vert \mathbb{E}\mathcal{P}_{\gamma_n, \mathbf{k}_n} \cdots \mathcal{P}_{\gamma_1, \mathbf{k}_1} \prod_{t=1}^s G^{0}_{a_tb_t} \right\vert+O(N^{-\zeta}),  
\end{equation} 
where 
\begin{equation} \label{indexassumption}
\mathbf{k}_1 \in \mathbb{N}^s, \  \ \mathbf{k}_2 \in \mathbb{N}^{s+ \vert \mathbf{k}_1 \vert}, \ \ \mathbf{k}_3 \in \mathbb{N}^{s+ \vert \mathbf{k}_1 \vert+ \vert \mathbf{k}_2 \vert}, \ \ldots, \ \text{ and } \ 5 \leq \vert \mathbf{k}_i \vert \leq \frac{2 \zeta}{\phi}.
\end{equation}
Again using (\ref{SECOND}), (\ref{INDEXBOUND}) and $s, \zeta = O(\varphi)$, we obtain that
\begin{equation} \label{teles_inequality3}
\left\vert \mathbb{E}\prod_{t=1}^s G^{\gamma_{\max}}_{a_tb_t} \right\vert  \leq   \left\vert \mathbb{E}\prod_{t=1}^s G^{0}_{a_tb_t} \right\vert + C_7^s \max_{\mathbf{k}, n}(N^{-2})^n(N^{-\phi/20})^{\sum_{i}\vert \mathbf{k}_i \vert}\sum_{\gamma_1, \cdots, \gamma_n}  \left\vert \mathbb{E}\mathcal{P}_{\gamma_n, \mathbf{k}_n} \cdots \mathcal{P}_{\gamma_1, \mathbf{k}_1} \prod_{t=1}^s G^{0}_{a_tb_t} \right\vert+O(N^{-\zeta}).  
\end{equation}  
for some constant $C_7>0$ depending on $C_6$. We note that the above estimate still holds if we replace some of the $G$ entries with $\overline G$ entries, since we have only used the absolute bounds for the relevant terms.

Now we apply (\ref{teles_inequality3}) to $G_{ab}\overline{G}_{ab}$ with $s=2$ and $\zeta=1$. 
Recall that $\tilde{X}$ is a bounded support matrix with $q= O (N^{-1/2} \log N)$. Then by (\ref{DIAGONAL}), we have with $\xi_1$-high probability, 
\begin{equation} \label{thm_diagonalreduce}
\vert \tilde{G}_{ab}  \vert \leq \varphi^{C_1} \left(\sqrt{\frac{\Im \, m_{2c}(z)}{N \eta}}+\frac{1}{N \eta} \right), \ \ a\ne b,
\end{equation} 
for $z\in S(c_1,C_0,C_1)$, where we used that
$$N^{-1/2} \le C\sqrt{\frac{\Im \, m_{2c}(z)}{N \eta}},$$
by (\ref{SQUAREROOTBEHAVIOR}). On the other hand, we have the trivial bound $\max_{a,b}|\tilde{G}_{ab}| \le CN$ on the bad event (see (\ref{eq_gbound})). Hence we can get the bound
$$\mathbb E\vert \tilde {G}_{ab}  \vert^2 \leq C\varphi^{2C_1} \left({\frac{\operatorname{Im} \, m_{2c}(z)}{N \eta}}+\frac{1}{(N \eta)^2} \right), \ \ a\ne b.$$
Again with (\ref{SQUAREROOTBEHAVIOR}), it is easy to check that the right-hand side is larger than $N^{-1}$. Thus the remainder term $O(N^{-\zeta})$ in (\ref{teles_inequality3}) is negligible.

It remains to handle the second term on the right-hand side of (\ref{teles_inequality3}). Let $\Phi(a_t, b_t)=\gamma_t$. Then we have
\begin{equation} \label{powerbound1}
\max_{\gamma_1, \cdots, \gamma_n: a, b \notin \cup_{1 \leq t \leq n} \{a_t,b_t\}} \left| \mathbb{E}\mathcal{P}_{\gamma_n, \mathbf{k}_n  } \cdots \mathcal{P}_{\gamma_1, \mathbf{k}_1} \left( \tilde{G}_{a b} \overline{\tilde{G}}_{ab}\right) \right| \leq C\varphi^{2C_1}\left(\frac{\operatorname{Im} m_{1c}(z)}{N \eta}+\frac{1}{(N \eta)^2}\right),
\end{equation}
since $\mathcal{P}_{\gamma_n, \mathbf{k}_n} \cdots \mathcal{P}_{\gamma_1, \mathbf{k}_1} \tilde{G}_{a_tb_t} \overline{\tilde{G}}_{a_tb_t}$ is a finite sum of the products of the matrix entries of $\tilde{G}$ and $\overline{\tilde{G}}$, and there are at least two off diagonal terms in each product. This bound immediately gives that
\begin{equation*}
\left(N^{-2}\right)^n \sum_{\gamma_1, \gamma_2, \cdots,\gamma_n} \left\vert \mathbb{E}\mathcal{P}_{\gamma_n, \mathbf{k}_n  } \cdots \mathcal{P}_{\gamma_1, \mathbf{k}_1}\left( \tilde{G}_{a b} \overline{\tilde{G}}_{a b}\right) \right\vert \leq \varphi^{C_2}\left(\frac{\operatorname{Im} m_{1c}(z)}{N \eta}+\frac{1}{(N \eta)^2}\right).
\end{equation*}
for some constant $C_2>2C_1$. Plug it into (\ref{teles_inequality3}), we conclude Lemma \ref{thm_largebound}. 
\end{proof}
%

\begin{proof}[Proof of Lemma \ref{lem_comgreenfunction}]
For simplicity, instead of (\ref{KEYEYEYEY}), we shall prove that
\begin{equation}
\vert \mathbb{E} \left(m_2(z)-m_{2c}(z) \right)^p| \leq \vert \mathbb{E} \left(\tilde m_2(z) - m_{2c}(z)\right)^p \vert + (Cp)^{Cp}(J^2+K + N^{-1})^p. \label{simpler_pf}
\end{equation}
The proof for (\ref{KEYEYEYEY}) is exactly the same but with slightly heavier notations (because we will only use the absolute bounds for relevant terms). 

Define a function $f(I,J)$ such that
\begin{equation} \label{linear_1}
\sum_{I,J} f(I,J)=1, \ \ f(I,J) \geq 0, \ \ I=(a_1,a_2,\cdots,a_s), \ \ J=(b_1,b_2,\cdots,b_s).
\end{equation}
Since $\mathcal{A}$ and $\mathcal{P}$ are independent of $a_t$ and $b_t$ ($1 \leq t \leq s$), we may consider a linear combination of (\ref{teles_inequality3}) with coefficients given by $f(I,J)$.
Moreover with (\ref{ONLYPROVE2}), we can extend (\ref{teles_inequality3}) to the product of $G-\Pi$ terms  for some constant $C_8>0$, i.e.
\begin{align}
& \left\vert \mathbb{E}\sum_{I,J}f(I,J)\prod_{t=1}^s \left(G - \Pi\right)_{a_tb_t} \right\vert \leq  \left\vert \mathbb{E}\sum_{I,J}f(I,J)\prod_{t=1}^s (\tilde G-\Pi)_{a_tb_t} \right\vert \nonumber\\
&  +(C_8)^s \max_{\mathbf{k}, n,\gamma}(N^{-\phi/20})^{\sum_{i}\vert \mathbf{k}_i \vert} \left\vert \mathbb{E}\sum_{I,J}f(I,J)\tilde {\mathcal{P}}_{\gamma_n, \mathbf{k}_n  } \cdots \tilde{ \mathcal{P}}_{\gamma_1, \mathbf{k}_1}\prod_{t=1}^s (\tilde G-\Pi)_{a_tb_t} \right\vert+ O(N^{-\zeta}). \label{KEY22}
\end{align}
If we take $a_t=b_t \in \mathcal I_2, \ s=\zeta=p$ and $f(I,J)= N^{-p} \prod \delta _{a_t,b_t}$, it is easy to check that
\begin{equation}
 \mathbb{E}\sum_{I,J}f(I,J)\prod_{t=1}^s (G^{\alpha}-\Pi)_{a_tb_t}= \mathbb{E}(m_2^{\alpha}-m_{2c})^p, \ \ \alpha=0, \ \gamma_{\max} .
\end{equation}
Now to conclude (\ref{simpler_pf}), it suffices to control the first term on the line (\ref{KEY22}). We consider the terms
\begin{flalign}\label{average_comparison1}
\tilde{\mathcal{P}}_{\gamma_n, \mathbf{k}_n  } \cdots \tilde{\mathcal{P}}_{\gamma_1, \mathbf{k}_1}\prod_{t=1}^p \left(\tilde{G}_{\mu_t\mu_t}-m_{2c}\right),
\end{flalign}
for $\mathbf k_1,\ldots, \mathbf k_n$ satisfying (\ref{indexassumption}).
By definition of $\tilde {\mathcal P}$, (\ref{average_comparison1}) is a sum of at most $C^{\sum \vert \mathbf{k}_i \vert}$ products of $\tilde{G}_{\mu\nu}$ and $(\tilde{G}_{\mu\mu}-m_{2c})$ terms, where the total number of $\tilde{G}_{\mu\nu}$ and $(\tilde{G}_{\mu\mu}-m_{2c})$ terms in each product is $\sum {\vert \mathbf{k}_i \vert}+p=O(\varphi^2)$. Due to the rough bound (\ref{5BOUNDT}), (\ref{average_comparison1}) is always bounded by $N^{O(\varphi^2)}$. Then with the assumptions that (\ref{KEYBOUNDS}) and (\ref{5BOUND1}) hold with $\xi_1$-high probability with $\xi_1\ge 3$, we see that the event that (\ref{KEYBOUNDS}) or (\ref{5BOUND1}) does not hold is negligible. 
Furthermore, for each product in (\ref{average_comparison1}) and any $1\le t \le p$, 
there are two $\mu_t$'s in the indices of $G$. These two $\mu_t$'s can only appear as (1) $(\tilde{G}_{\mu_t\mu_t}-m_{2c})$ in the product, or (2) $G_{\mu_t a} G_{b \mu_t}$, where $a, b$ come from some $\gamma_k$ and $\gamma_{l}$ via $\tilde{\mathcal{P}}$ (see Definition \ref{def_operator2}). Then after averaging over  $N^{-p} \sum_{\mu_1,\cdots, \mu_p}$, this term becomes (1) $\tilde{m}_2-m_{2c}$, which is bounded by $K$ by (\ref{KEYBOUNDS}), or (2) $N^{-1}\sum_{\mu_t} G_{\mu_t a} G_{b, \mu_t}$, which is bounded by $J^2 + CN^{-1}$ by (\ref{KEYBOUNDS}). Here for the $\mu_t=a$ or $b$ terms in case (2), we control the $G$ factors by $C$ using (\ref{5BOUND1}). In sum, for any fixed $\gamma_1,\ldots,\gamma_n$, $\mathbf k_1,\ldots, \mathbf k_n$, we have proved that
\begin{equation}
 \left\vert \frac{1}{N^p} \mathbb{E}\tilde{\mathcal{P}}_{\gamma_n, \mathbf{k}_n  } \cdots \tilde{\mathcal{P}}_{\gamma_1, \mathbf{k}_1}\prod_{t=1}^s \left(\tilde{G}_{\mu_t\mu_t}-m_{2c}\right) \right\vert \leq C^{\sum {\vert \mathbf{k}_i \vert+p}}\left(J^2+K + N^{-1}\right)^p . \label{KEY444}
\end{equation}
Together with (\ref{KEY22}), we conclude (\ref{simpler_pf}).
\end{proof}


Recall that Lemma \ref{lem_comgreenfunction} leads to the proof of Theorem \ref{thm_largerigidity}. Finally we prove Lemma \ref{lem_compdiffsupport} with Theorem \ref{thm_largerigidity}.

\begin{proof}[Proof of Lemma \ref{lem_compdiffsupport}]
For simplicity, we only prove (\ref{BDBD}). The proof for (\ref{BDBD1}) is similar. By (\ref{eq_gsq1}), we have
\begin{equation}\label{sum_formula}
\| \mathcal{G}_2(z) \|^2_{HS}= \sum_{\mu,\nu}|G_{\mu\nu}|^2=\frac{N\Im\, m_2(z)}{\eta}.
\end{equation}
Hence, it is equivalent to show that
\begin{equation}
\left\vert \mathbb{E}F\left(\eta^2 \sum_{\mu,\nu} G_{\mu\nu}\overline G_{\mu\nu}\right) -\mathbb{E}F\left(\eta^2 \sum_{\mu,\nu} {\tilde G}_{\mu\nu}\overline {\tilde G}_{\mu\nu}\right) \right\vert \leq N^{-\phi+C_4 \epsilon} , 
\end{equation}
for $z= E+i\eta$ with $E\in I_\epsilon$ and $\eta= N^{-2/3-\epsilon}$. Corresponding to the notations in (\ref{R}), we denote
\begin{equation}
x^S: = \eta^2 \sum_{\mu,\nu} S_{\mu\nu} \overline{S}_{\mu\nu}, \ \ x^R:=\eta^2 \sum_{\mu,\nu} R_{\mu\nu} \overline{R}_{\mu\nu}, \ \ x^T:=\eta^2 \sum_{\mu,\nu} T_{\mu\nu} \overline{T}_{\mu\nu}.
\end{equation}
Applying (\ref{sum_formula}) to $S,T$ and using (\ref{NEWMPBOUNDS}) and (\ref{SQUAREROOTBEHAVIOR}), we get that with high probability
\begin{equation}
\max_{\gamma} \left\{ \left\vert x^S \right\vert+\left\vert x^T \right\vert \right\} \leq N^{C\epsilon}.
\label{578}
\end{equation}
Since the rank of $H^\gamma - Q$ is at most 2, by Cauchy interlacing theorem, we have that
\begin{equation}
\left\vert \text{Tr} \, S - \text{Tr} \, R \right\vert \leq C \eta^{-1}. \label{577}
\end{equation}
Together with (\ref{578}), we also get that
\begin{equation}
\max_{\gamma} \left\vert x^R \right\vert \leq N^{C\epsilon} \ \text{ with high probability}. \label{578T}
\end{equation}
By (\ref{DIAGONAL}), (\ref{5BOUND2}) and the expansion (\ref{RESOLVENTEXPANSION2}), we get that with high probability,
\begin{equation}
\max_{\gamma}\left\{\vert S_{\mu\nu} \vert+\vert R_{\mu\nu} \vert \right\} \leq N^{-\phi+C\epsilon}+C\delta_{\mu\nu}. \label{579}
\end{equation}
Moreover, by (\ref{5BOUNDT}) we have the trivial bounds
\begin{equation}\label{bad_high}
\left\vert x^S \right\vert + \left\vert x^R \right\vert =O(N^2), \ \ \max_{\mu,\nu}\{\vert S_{\mu\nu} \vert+\vert R_{\mu\nu} \vert \} =O(N),
\end{equation}
on the bad event. Since the bad event holds with exponentially small probability, we can ignore it in the proof. 

Applying the Lindeberg replacement strategy, we get that
\begin{equation}\label{taylorexp1}
\mathbb{E}F\left(\eta^2 \sum_{\mu,\nu} G_{\mu\nu}\overline G_{\mu\nu}\right) -\mathbb{E}F\left(\eta^2 \sum_{\mu,\nu} {\tilde G}_{\mu\nu}\overline {\tilde G}_{\mu\nu}\right) =\sum_{\gamma=1}^{\gamma_{\max}}\left[\mathbb{E}F\left(x^S\right)-\mathbb{E}F\left(x^T\right)\right].
\end{equation}
From the Taylor expansion, we have
\begin{equation}\label{taylorexpansion}
F\left(x^S\right)-F\left(x^R\right)=\sum_{s=1}^2 \frac{1}{s !} F^{(s)}\left(x^R\right)\left(x^S-X^R\right)^s+\frac{1}{3!}F^{(3)}\left(\zeta_S\right)\left(x^S-x^R\right)^3,
\end{equation}
where $\zeta_S$ lies between $x^S$ and $x^R$. We have a similar expansion for $F\left(x^T\right)- F\left(x^R\right)$ with $\zeta_S$ replaced with $\zeta_T$. Let $\Phi(i,\mu)=\gamma$. We perform the expansion (\ref{RESOLVENTEXPANSION}) and use (\ref{graphical}) to get that 
\begin{equation}
S_{a_tb_t}=\sum_{0 \leq k \leq m } (-\sqrt{\sigma_i}x_{i\mu})^{k}\mathcal{P}_{k}R_{a_tb_t}+O(N^{- m\phi}). 
\end{equation}
Using this expansion and bound (\ref{5BOUND1}), we have that with $\xi_1$-high probability,
\begin{equation}\label{Bigsum}
\prod_{t=1}^s S_{a_tb_t}= \sum_{ 0 \leq k \leq ms} \ \sum_{\mathbf{k} \in I_{m,k}^s} \left(\mathcal{P}_{\mathbf{k}} \prod_{t=1}^s R_{a_tb_t}\right)(-\sqrt{\sigma_i}x_{i\mu})^{k_t}+O\left((2C_6)^s N^{-\xi\phi}\right),
\end{equation}
where 
\begin{equation}\label{def_Iab}
\mathbf{k}:=(k_1,\cdots,k_s), \ \ I^s_{m,k}=\left\{\mathbf{k} \in \mathbb{N}^s: 0 \leq k_i \leq m, \sum k_i=k \right\}.
\end{equation}
From the above definition, we have the rough bound 
\begin{equation}
\vert I_{m,k}^s \vert \leq s^k . \label{BBBBBBBBBBB}
\end{equation}
By Lemma \ref{propsotion_poperator} and (\ref{BBBBBBBBBBB}), the $k > m$ terms in (\ref{Bigsum}) can be bounded by
\begin{align*}
\left\vert \sum_{k > m} \ \sum_{\mathbf{k} \in I_{m,k}^s} \left(\mathcal{P}_{\mathbf{k}} \prod_{t=1}^s R_{a_tb_t}\right)(-\sqrt{\sigma_i}x_{i\mu})^{k_t} \right\vert \leq  \sum_{k > m} s^{k} (2C_6)^{k+s+1} \left(CN^{-\phi}\right)^k = O(s^{m} C^{m+s} N^{-m\phi}), 
\end{align*}
with $\xi_1$-high probability. Hence with $\xi_1$-high probability,
\begin{equation}
\prod_{t=1}^s S_{a_tb_t}=\prod_{t=1}^s R_{a_tb_t}+\sum_{1 \leq k \le m} (-\sqrt{\sigma_i}x_{i\mu})^{k} \left( \sum_{\mathbf{k} \in I_{m,k}^s} \mathcal{P}_{\mathbf{k}}\prod_{t=1}^s R_{a_tb_t}\right)+O\left(s^{m} C^{m+s} N^{-m\phi}\right). \label{REDUCED1}
\end{equation}
Similarly, we also have
\begin{equation}
\prod_{t=1}^s T_{a_tb_t}=\prod_{t=1}^s R_{a_tb_t}+\sum_{1 \leq k \le m} (-\sqrt{\sigma_i}\tilde{x}_{i\mu})^{k} \left( \sum_{\mathbf{k} \in I_{m,\alpha}^s} \mathcal{P}_{\mathbf{k}}\prod_{t=1}^s R_{a_tb_t}\right)+O\left(s^{m} C^{m+s} N^{-m\phi}\right). \label{REDUCED2}
\end{equation}
Again we can replace some of the resolvent entries with its complex conjugate by making some slight modifications to the notations.
Hence using (\ref{REDUCED1}) and (\ref{REDUCED2}) with $s=2$ and $m:=3/\phi$, we obtain that
\begin{equation} \label{rsequation}
x^S=x^R+\sum_{1 \leq k \leq 3/\phi}\left(\sum_{\mathbf{k} \in I_{{3}/{\phi},k}^2}\eta^2 \sum_{\mu,\nu} \mathcal{P}_{\gamma, \mathbf{k}}(R_{\mu\nu}\overline{R}_{\mu\nu}) \right)(-\sqrt{\sigma_i}x_{i\mu})^{k}+O(CN^{-3}),
\end{equation}
and
\begin{equation} \label{rsequation2}
x^T=x^R+\sum_{1 \leq k \leq 3/\phi}\left(\sum_{\mathbf{k} \in I_{{3}/{\phi},k}^2}\eta^2 \sum_{\mu,\nu} \mathcal{P}_{\gamma, \mathbf{k}}(R_{\mu\nu}\overline{R}_{\mu\nu})\right)(-\sqrt{\sigma_i}\tilde{x}_{i\mu})^{k}+O(CN^{-3}),
\end{equation}
with high probability.
To control the second term in (\ref{rsequation}), we need the following lemma.
\begin{lem} \label{lem_control} For any fixed $\mathbf{k} \neq 0$, $\mathbf{k} \in I^2_{3/\phi, k}$, and $p=O(1)$ with $p \in 2 \mathbb{Z}$, we have 
\begin{equation}\label{lem_skip}
\mathbb{E} \left\vert \sum_{\mu,\nu} \mathcal{P}_{\gamma,\mathbf{k}}  \left(R_{\mu\nu}\overline{R}_{\mu\nu}\right) \right\vert^p \leq \left(N^{1+C\epsilon}\right)^p.
\end{equation}
\end{lem}
\begin{proof}
This is (6.89) of \cite{LY}, we can repeat the proof there with minor modifications. In fact, its proof is very similar to the one of Lemma \ref{lem_comgreenfunction} given above. 
\end{proof} 
Given (\ref{lem_skip}), with Markov inequality we find that for any fixed $\mathbf{k} \neq 0$, $\mathbf{k} \in I^2_{3/\phi, k}$,
\begin{equation} \label{productbound}
\vert \mathcal{P}_{\gamma, \mathbf{k}} x^R \vert := \left\vert \eta^2 \sum_{\mu,\nu} \mathcal{P}_{\gamma,\mathbf{k}}\left(R_{\mu\nu}\overline{R}_{\mu\nu} \right)\right\vert \leq N^{-1/3+C\epsilon},
\end{equation}
holds with probability with $1-N^{-A}$ for any fixed $A>0$, where we used that $\eta=N^{-2/3-\epsilon}$. 
Combining (\ref{rsequation}), (\ref{productbound}) and (\ref{conditionA3}), we see that there exists a constant $C>0$ such that
\begin{equation} \label{taylorbound_1}
\mathbb{E} \vert x^S-x^R \vert^3 \leq N^{-5/2+C\epsilon},
\end{equation}
for sufficiently large $N$ independent of $\gamma$, where we used the bound (\ref{578}) on the bad event with probability $\le N^{-A}$. Since $\zeta_S$ is between $x^S$ and $x^R$, we have $|\zeta_S|\le N^{C\epsilon}$ with high probability by (\ref{578}). Together with (\ref{taylorbound_1}) and the assumption (\ref{FCondtion}), we get
\begin{equation} \label{taylorbound2}
\left\vert \sum_{\gamma=1}^{\gamma_{\max}} \mathbb{E}\left[F^{(3)}(\zeta_S)\left(x^S-x^R\right)^3\right] \right\vert \leq N^{-\phi+C\epsilon},
\end{equation}
for some $C_4>0$. We have a similar estimate for $\mathbb{E}\left[F^{(3)}(\zeta_T)\left(x^T-x^R\right)^3 \right]$. 

Now it only remains to deal with the first term on the right-hand side of (\ref{taylorexpansion}). Using (\ref{rsequation}), (\ref{rsequation2}) and the fact that the first four moments of $x_{i\mu}$ and $\tilde{x}_{i\mu}$ match, we obtain that
\begin{align}
& \left \vert \mathbb{E}\left[F^{(s)}(x^R)\left(x^S-x^R\right)^s\right]-\mathbb{E}\left[F^{(s)}(x^R)\left(x^T-x^R\right)^s \right] \right \vert  \nonumber \\
& \leq \left| \sum_{k=5}^{{9}/{\phi}} \ \sum_{\sum_{t=1}^s \vert \mathbf{k}_t \vert=k} \sum_{\mathbf{k}_t \in I_{{3}/{\phi},k}^{2s}}\mathbb{E} \prod_{t=1}^s \left(\mathcal{P}_{\gamma, \mathbf{k}_t }x^R\right)\right| \left(\left\vert \mathbb{E}(-\sqrt{\sigma_i}x_{i \mu})^{k} \right\vert+ \left\vert \mathbb{E}(\sqrt{\sigma_i}\tilde{x}_{i \mu})^{k} \right\vert \right)+O(CN^{-3}).
\end{align}
Recall that (\ref{conditionA3}) holds for $x_{\i\mu}$ and $\tilde x_{i\mu}$, $x_{i\mu}$ has support bounded by $N^{-\phi}$, and $\tilde x_{i\mu}$ has support bounded by $ O(N^{-1/2}\log N)$. Then it is easy to check that $ \vert \mathbb{E}(\tilde{x}_{i \mu})^{k} \vert \leq (\log N)^{C} N^{-5/2}$ and  $ \vert \mathbb{E}({x}_{i \mu})^{k} \vert \leq (\log N)^{C} N^{-2-\phi}$ for $k\ge 5$. Using (\ref{productbound}), we obtain that
\begin{equation} \label{taylorbound3}
\left\vert \mathbb{E}\left[F^{(s)}(x^R)(x^S-x^R)^s\right]-\mathbb{E}\left[F^{(s)}(x^R)(x^T-x^R)^s\right] \right\vert \leq N^{-2-\phi+C_4\epsilon}, \ 1 \leq s \leq 2
\end{equation}
Together with (\ref{taylorexp1}), (\ref{taylorexpansion}) and (\ref{taylorbound2}), we conclude the proof.
\end{proof}

\begin{appendix}
\section{Proof of Lemma \ref{lem_small}} \label{appendix1}

A major part of this appendix is devoted to the proof of the entrywise local law (\ref{DIAGONAL}) and the averaged local law (\ref{MPBOUNDS}). The other results of Lemma \ref{lem_small} are mostly consequences of (\ref{MPBOUNDS}) and (\ref{DIAGONAL}), and we will briefly describe their proof at the end of this appendix. 
We will basically follow the approach in \cite{KY2}, but modify some arguments under different assumptions in this paper. Throughout this section, we denote the spectral parameter by $z=E+i\eta$.

\subsection{Basic tools}
In this subsection, we collect some tools that will be used in the proof. For simplicity, we denote $Y:=DX$. 

\begin{defn}[Minors]
For $\mathbb T \subseteq \mathcal I$, we define the minor $H^{(\mathbb T)}:=(H_{ab}:a,b \in \mathcal I\setminus \mathbb T)$ obtained by removing all rows and columns of $H$ indexed by $a\in \mathbb T$. Note that we keep the names of indices of $H$ when defining $H^{(\mathbb T)}$, i.e. $(H^{(\mathbb{T})})_{ab}=\mathbf{1}_{ \{a,b \notin \mathbb{{T}}\}} H_{ab}$. Correspondingly, we define the Green function 
$$G^{(\mathbb T)}:=\left(H^{(\mathbb T)}\right)^{-1}= \left( {\begin{array}{*{20}c}
   { z\mathcal G_1^{(\mathbb T)}} & \mathcal G_1^{(\mathbb T)} Y^{(\mathbb T)}  \\
   {\left(Y^{(\mathbb T)}\right)^*\mathcal G_1^{(\mathbb T)}} & { \mathcal G_2^{(\mathbb T)} }  \\
\end{array}} \right)= \left( {\begin{array}{*{20}c}
   { z\mathcal G_1^{(\mathbb T)}} & Y^{(\mathbb T)}\mathcal G_2^{(\mathbb T)}   \\
   {\mathcal G_2^{(\mathbb T)}}\left(Y^{(\mathbb T)}\right)^* & { \mathcal G_2^{(\mathbb T)} }  \\
\end{array}} \right),$$
and the partial traces
$$m_1^{(\mathbb T)}:=\frac{1}{M}{\rm{Tr}}\, \mathcal G_1^{(\mathbb T)} = \frac{1}{Mz}\sum_{i\notin \mathbb T}G_{ii}^{(\mathbb T)},\ \ m_2^{(\mathbb T)}:=\frac{1}{N}{\rm{Tr}}\, \mathcal G_2^{(\mathbb T)} = \frac{1}{N}\sum_{\mu \notin \mathbb T}G_{\mu\mu}^{(\mathbb T)}.$$ 
We will abbreviate $(\{a\})\equiv (a)$, $(\{a, b\})\equiv (ab)$, and
$$\sum_{a\notin \mathbb T} \equiv \sum_{a}^{(\mathbb T)} \ , \ \ \sum_{a,b\notin \mathbb T} \equiv \sum_{a,b}^{(\mathbb T)}\ .$$
\end{defn}


\begin{lem}{(Resolvent identities).}

\begin{itemize}
\item[(i)]
For $i\in \mathcal I_1$ and $\mu\in \mathcal I_2$, we have
\begin{equation}
\frac{1}{{G_{ii} }} =  - 1 - \left( {YG^{\left( i \right)} Y^*} \right)_{ii} ,\ \ \frac{1}{{G_{\mu \mu } }} =  - z  - \left( {Y^*  G^{\left( \mu  \right)} Y} \right)_{\mu \mu }.\label{resolvent2}
\end{equation}

 \item[(ii)]
 For $i\ne j \in \mathcal I_1$ and $\mu \ne \nu \in \mathcal I_2$, we have
\begin{equation}
G_{ij}   = G_{ii} G_{jj}^{\left( i \right)} \left( {YG^{\left( {ij} \right)} Y^* } \right)_{ij},\ \ G_{\mu \nu }  = G_{\mu \mu } G_{\nu \nu }^{\left( \mu  \right)} \left( {Y^*  G^{\left( {\mu \nu } \right)} Y} \right)_{\mu \nu }. \label{resolvent3}
\end{equation}
For $i\in \mathcal I_1$ and $\mu\in \mathcal I_2$, we have
\begin{equation}
G_{i\mu } = G_{ii} G_{\mu \mu }^{\left( i \right)} \left( { - Y_{i\mu }  +  {\left( {YG^{\left( {i\mu } \right)} Y} \right)_{i\mu } } } \right), \ \ G_{\mu i}  = G_{\mu \mu } G_{ii}^{\left( \mu  \right)} \left( { - Y_{\mu i}^*  + \left( {Y^*  G^{\left( {\mu i} \right)} Y^*  } \right)_{\mu i} } \right)\label{resolvent6}.
\end{equation}

 \item[(iii)]
 For $a \in \mathcal I$ and $b, c \in \mathcal I \setminus \{a\}$,
\begin{equation}
G_{bc}^{\left( a \right)}  = G_{bc}  - \frac{G_{ba} G_{ac}}{G_{aa}}, \ \ \frac{1}{{G_{bb} }} = \frac{1}{{G_{bb}^{(a)} }} - \frac{{G_{ba} G_{ab} }}{{G_{bb} G_{bb}^{(a)} G_{aa} }}. \label{resolvent8}
\end{equation}

 \item[(iv)]
All of the above identities hold for $G^{(\mathbb T)}$ instead of $G$ for $\mathbb T \subset \mathcal I$.
\end{itemize}
\label{lemm_resolvent}
\end{lem}
\begin{proof}
All these identities can be proved using Schur's complement formula. The reader can refer to, for example, \cite[Lemma 4.4]{KY2}.
\end{proof}

\begin{lem}\label{Ward_id}
Fix constants $c_0,C_0, C_1>0$. The following estimates hold uniformly for any $z\in S(c_0,C_0,C_1)$:
\begin{equation}
\left\| G \right\| \le C\eta ^{ - 1} ,\ \ \left\| {\partial _z G} \right\| \le C\eta ^{ - 2}. \label{eq_gbound}
\end{equation}
Furthermore, we have the following identities:
\begin{align}
\sum\limits_{\mu  \in \mathcal I_2 } {\left| {G_{\nu \mu } } \right|^2 } = \sum\limits_{\mu  \in \mathcal I_2 } {\left| {G_{\mu \nu} } \right|^2 }  = \frac{{\Im \, G_{\nu\nu} }}{\eta }, \ \ &\sum\limits_{i \in \mathcal I_1 }  \left| {G_{j i} } \right|^2 = \sum\limits_{i \in \mathcal I_1 }  \left| {G_{ij} } \right|^2  = \frac{|z|^2}{\eta}\Im\left(\frac{G_{jj}}{z}\right) , \label{eq_gsq1} \\
\sum\limits_{i \in \mathcal I_1 } {\left| {G_{\mu i} } \right|^2 } = \sum\limits_{i \in \mathcal I_1 } {\left| {G_{i\mu} } \right|^2 } = {G}_{\mu \mu}  + \frac{\bar z}{\eta} \Im \, G_{\mu\mu} , \ \ &\sum\limits_{\mu \in \mathcal I_2 } {\left| {G_{i \mu} } \right|^2 } = \sum\limits_{\mu \in \mathcal I_2 } {\left| {G_{\mu i} } \right|^2 } =  \frac{{G}_{ii}}{z}  + \frac{\bar z}{\eta} \Im\left(\frac{{G_{ii} }}{z}\right) . \label{eq_gsq3} 
 \end{align}
All of the above estimates remain true for $G^{(\mathbb T)}$ instead of $G$ for any $\mathbb T \subseteq \mathcal I$. 
\label{lemma_Im}
\end{lem}
\begin{proof}
These estimates and identities can be proved through simple calculations using (\ref{green2}), (\ref{spectral1}) and (\ref{spectral2}). We refer the reader to \cite[Lemma 4.6]{KY2} and \cite[Lemma 3.5]{XYY}.
\end{proof}

\begin{lem}
Fix constants $c_0,C_0, C_1>0$. For any $\mathbb T \subseteq \mathcal I$, the following bounds hold uniformly in $z\in S(c_0,C_0,C_1)$:
\begin{equation}\label{m_T}
\left| {m_2  - m_2^{\left( \mathbb T \right)} } \right| \le \frac{{2\left| \mathbb T \right|}}{{N\eta }}, 
\end{equation}
and 
\begin{equation}\label{m11_T}
\left| {\frac{1}{N}\sum_{i=1}^M \sigma_i \left(G_{ii}^{(\mathbb T)} - G_{ii}\right)} \right| \le \frac{{C\left| \mathbb T \right|}}{{N\eta }}, 
\end{equation}
where $C$ is a constant depending only on $\tau$.
\end{lem}
\begin{proof}
For $\mu\in\mathcal I_2$, we have
\begin{align}
\left|m_2-m_2^{(\mu)}\right|& =\frac{1}{N}\left|\sum_{\nu\in\mathcal I_2}  \frac{G_{\nu\mu}G_{\mu\nu}}{G_{\mu\mu}}\right| \le \frac{1}{N|G_{\mu\mu}|} \sum_{\nu\in\mathcal I_2} |G_{\nu\mu}|^2 = \frac{\Im\, G_{\mu\mu}}{N\eta |G_{\mu\mu}|} \le \frac{1}{N\eta},  \label{rough_boundmi}
\end{align}
where in the first step we used (\ref{resolvent8}), in the second and third steps the equality (\ref{eq_gsq1}). Similarly, using (\ref{resolvent8}) and (\ref{eq_gsq3}) we get
\begin{align*}
\left|m_2 -m_2^{(i)}\right| & = \frac{1}{N}\left|\sum_{\nu \in\mathcal I_2}\frac{G_{\nu i}G_{i\nu}}{G_{ii}}\right| \le \frac{1}{N|G_{ii}|} \left( \frac{{G}_{ii}}{z}  + \frac{\bar z}{\eta} \Im\left(\frac{{G_{ii} }}{z}\right)\right)   \le \frac{2}{N\eta}.
\end{align*}
Then we can prove (\ref{m_T}) by induction on the indices in $\mathbb T$. The proof for (\ref{m11_T}) is similar except that one need to use the assumption (\ref{assm3}).
\end{proof}

The following large deviation bounds for bounded supported random variables are proved in \cite[Lemma 3.8]{EKYY1}. 

\begin{lem}\label{largederivation}
Let $(x_i)$, $(y_i)$ be independent families of centered and independent random variables, and $(A_i)$, $(B_{ij})$ be families of deterministic complex numbers. Suppose the entries $x_i$ and $y_j$ have variance at most $N^{-1}$ and satisfies the bounded support condition (\ref{eq_support}) with $q\le N^{-\epsilon}$ for some $\epsilon>0$. Then for any fixed $\xi>0$, the followings hold with $\xi$-high probability:
\begin{align}
\left\vert \sum_i A_i x_i \right\vert & \leq \varphi^{\xi}\left[q \max_{i} \vert A_i \vert+ \frac{1}{\sqrt{N}}\left(\sum_i |A_i|^2 \right)^{1/2} \right], \\
\left\vert \sum_{i,j} x_i B_{ij} y_j \right\vert & \leq \varphi^{2\xi}\left[q^2 B_d  + qB_o + \frac{1}{N}\left(\sum_{i\ne j} |B_{ij}|^2\right)^{{1}/{2}} \right], \\
\left\vert \sum_{i} \bar x_i B_{ii} x_i - \sum_{i} (\mathbb E|x_i|^2) B_{ii}  \right\vert & \leq \varphi^{\xi}q B_d  ,\ \ \left\vert \sum_{i\ne j} \bar x_i B_{ij} x_j \right\vert  \leq \varphi^{2\xi}\left[qB_o + \frac{1}{N}\left(\sum_{i\ne j} |B_{ij}|^2\right)^{{1}/{2}} \right],
\end{align}
where
$$B_d:=\max_{i} |B_{ii} | , \ \ B_o:= \max_{i\ne j} |B_{ij}|.$$ 
\end{lem}  

%


Finally, we have the following lemma, which is a consequence of the Assumption \ref{assm_big2}.
\begin{lem}\label{lem_assm3}
There exists constants $c_0, \tau' >0$ such that 
\begin{equation}\label{assumption3}
|1+m_{2c}(z)\sigma_k |\ge \tau',
\end{equation}
for all $z \in S(c_0,C_0,C_1)$ and $1\le k \le M$.
\end{lem}
\begin{proof}
By Assumption \ref{assm_big2} and the fact $m_{2c}(\lambda_r) \in (-\sigma_1^{-1}, 0)$, we have
$$\left| 1+ \sigma_k m_{2c}(\lambda_r) \right| \ge \tau,  \ \ 1\le k \le M.$$
Applying (\ref{SQUAREROOT}) to the Stieltjes transform
\begin{equation}\label{Stj_app}
m_{2c}(z):=\int_{\mathbb R} \frac{\rho_{2c}(dx)}{x-z},
\end{equation}
we can verify that $m_{2c}(z) \sim \sqrt{z-\lambda_r}$ for $z$ close to $\lambda_r$. Hence if $\kappa+\eta \le 2c_0$ for some sufficiently small $c_0$, we have
$$\left| 1+ \sigma_k m_{2c}(z) \right| \ge \tau/2.$$
On the other hand, if $\eta \ge c_0$, there exists $\tau'$ depending on $c_0$ such that
$$\left| 1+ \sigma_k m_{2c}(z) \right| \ge \Im\, m_{2c}(z) \ge \tau' $$
by (\ref{SQUAREROOTBEHAVIOR}). Finally, it remains to consider the case $\eta \le c_0$ and $E-\lambda_r \ge c_0$. In fact, for $\eta=0$ and $E\ge \lambda_r + c_0$, it is easy to see that $m_{2c}'(E)\ge 0$ with the formula (\ref{Stj_app}). Hence we have
$$\left| 1+ \sigma_k m_{2c}(E) \right| \ge \left| 1+ \sigma_k m_{2c}(\lambda_r + c_0) \right| \ge \tau/2, \ \ E\ge \lambda_r + c_0.$$
Using (\ref{Stj_app}) again, we can verify that $|m_{2c}'(z)| = O(1)$. So if we choose $c_0$ sufficiently small, we have
$$\left| 1+ \sigma_k m_{2c}(E+i\eta) \right| \ge \frac{1}{2}\left| 1+ \sigma_k m_{2c}(E) \right| \ge \tau/4$$
for $E\ge \lambda_r + c_0$ and $\eta \le c_0$.
\end{proof}

\subsection{Proof of the local laws}

Throughout this section, we fix $\xi_1\ge 3$. Our goal is to prove that $G$ is close to $\Pi$ in the sense of entrywise and averaged local laws. Hence it is convenient to introduce the following random control parameters.

\begin{defn}[Control parameters]
We define the entrywise and averaged errors
\begin{equation}\label{eqn_randomerror}
\Lambda : = \mathop {\max }\limits_{a,b \in \mathcal I} \left| {\left( {G - \Pi } \right)_{ab} } \right|,\ \ \Lambda _o : = \mathop {\max }\limits_{a \ne b \in \mathcal I} \left| {G_{ab} } \right|, \ \ \theta:= |m_2-m_{2c}| .
\end{equation}
Moreover, we define the random control parameter
\begin{equation}\label{eq_defpsitheta}
\Psi _\theta  : = \sqrt {\frac{{\Im \, m_{2c}  + \theta }}{{N\eta }}} + \frac{1}{N\eta},
\end{equation}
and the deterministic control parameter
\begin{equation}\label{eq_defpsi}
\Psi := \sqrt {\frac{\Im\, m_{2c} }{{N\eta }} } + \frac{1}{N\eta}.
\end{equation}
\end{defn}

\begin{rem} 
By definition, we trivially have $\theta=O(\Lambda)$. Also by (\ref{barm}), we immediately get that
$$|m_1 - m_{1c}|=\theta d.$$
\end{rem}

In analogy to \cite[Section 3]{EKYY1} and \cite[Section 5]{KY2}, we introduce the $Z$ variables
\begin{equation*}
  Z_{a}^{(\mathbb T)}:=(1-\mathbb E_{a})\left(G_{aa}^{(\mathbb T)}\right)^{-1}, \ \ a\notin \mathbb T,
\end{equation*}
where $\mathbb E_{a}[\cdot]:=\mathbb E[\cdot\mid H^{(a)}],$ i.e. it is the partial expectation over the randomness of the $a$-th row and column of $H$. By (\ref{resolvent2}), we have
\begin{equation}\label{Zi}
Z_i = (\mathbb E_{i} - 1) \left( {YG^{\left( i \right)} Y^*} \right)_{ii} = \sigma_i \sum_{\mu ,\nu\in \mathcal I_2} G^{(i)}_{\mu\nu} \left(\frac{1}{N}\delta_{\mu\nu} - X_{i\mu}X_{i\nu}\right),
\end{equation}
and 
\begin{equation}\label{Zmu}
Z_\mu = (\mathbb E_{\mu} - 1) \left( {Y^*  G^{\left( \mu  \right)} Y} \right)_{\mu \mu }= \sum_{i,j \in \mathcal I_1} \sqrt{\sigma_i \sigma_j}G^{(\mu)}_{ij} \left(\frac{1}{N} \delta_{ij} - X_{i\mu}X_{j\mu}\right).
\end{equation}
The following estimate plays a key role in the proof of local laws.

\begin{lem}\label{Z_lemma}
Let $c_0>0$ be sufficiently small and fix $C_0, C_1,\xi >0$. Define the $z$-dependent event $\Xi(z):=\{\Lambda(z) \le (\log N)^{-1}\}$. Then there exists $C>0$ such that the following estimates hold for all $a\in \mathcal I$ and $z\in S(c_0,C_0,C_1)$ with $\xi$-high probability:
\begin{align}
{\mathbf 1}(\Xi)\left(\Lambda_o + |Z_{a}|\right) \le C\varphi^{2\xi}\left(q+\Psi_\theta\right), \label{Zestimate1}
\end{align}
and 
\begin{align}
{\mathbf 1}\left(\eta \ge 1 \right)\left(\Lambda_o + |Z_{a}|\right) \le C\varphi^{2\xi}\left(q+\Psi_\theta\right). \label{Zestimate2}
\end{align}
\end{lem}
\begin{proof}
 Applying the large deviation Lemma \ref{largederivation} to $Z_{i}$ in (\ref{Zi}), we get that on $\Xi$,
\begin{align}\label{estimate_Zi}
\left| Z_{i}\right| \le & C \varphi^{2\xi} \left[ q +\frac{1}{N} \left( \sum_{\mu, \nu} {\left| G_{\mu\nu}^{(i)}  \right|^2 }  \right)^{1/2}  \right] \nonumber\\
= & C \varphi^{2\xi} \left[ q + \frac{1}{N}\left( {\sum_\mu \frac{{\mathop{\rm Im}\nolimits}\, G_{\mu\mu}^{(i)} }{\eta} } \right)^{1/2}\right] = C \varphi^{2\xi} \left[ q + \sqrt { \frac{ \Im\, m_2^{(i)}  } {N\eta} }\right]
\end{align}
holds with $\xi$-high probability, where we used (\ref{assm3}), (\ref{eq_gsq1}) and the fact that $\max_{a,b}|G_{ab}|= O(1)$ on event $\Xi$. Now using the bound (\ref{m_T}) and the definitions (\ref{eqn_randomerror}), (\ref{eq_defpsitheta}), we get that
\begin{align}\label{m2psi}
\sqrt{\frac{{\Im\, m_2^{(i)} }}{N\eta} } = \sqrt {\frac{{\Im\,m_{2c}  + \Im ( {m_2^{(i)}  - m_2 }) + \Im ( {m_2  - m_{2c} } )}}{{N\eta }}}  \le C \Psi _\theta .
\end{align}
Together with (\ref{estimate_Zi}), we conclude that ${\mathbf 1}(\Xi) |Z_{i}|  \le C\varphi^{2\xi}\left(q+\Psi_\theta\right)$ with $\xi$-high probability. Similarly, we can prove the same estimate for ${\mathbf 1}(\Xi) |Z_{\mu}| $. In the proof, we also need to use (\ref{barm}) and
$$\Im \left( -\frac{d-1}{z}\right) = O(\eta) = O(\Im\, m_{2c}(z)).$$
If $\eta\ge 1$, we always have $\max_{a,b}|G_{ab}|= O(1)$ by (\ref{eq_gbound}). Then repeating the above proof, we obtain that ${\mathbf 1}(\eta\ge 1) |Z_{a}|  \le C\varphi^{2\xi}\left(q+\Psi_\theta\right)$ with $\xi$-high probability.

Similarly, using (\ref{resolvent3}) and Lemmas \ref{Ward_id}-\ref{largederivation}, we can prove that
\begin{equation}\label{2blocks}
{\mathbf 1}(\Xi) \left( |G_{ij}| + |G_{\mu\nu}|\right) \le C\varphi^{2\xi}\left(q+\Psi_\theta\right), 
\end{equation}
holds uniformly for $i\ne j$ and $\mu\ne \nu$ with $\xi$-high probability. It remains to prove the bound for $G_{i\mu}$ and $G_{\mu i}$. Using (\ref{resolvent6}), the bounded support condition (\ref{eq_support}) for $X_{i\mu}$, the bound $\max_{a,b}|G_{ab}|= O(1)$ on $\Xi$, and Lemma \ref{largederivation}, we get that with $\xi$-high probability,
\begin{align}
\left|G_{i\mu }\right| & \le C\left( q +  \left| \sum_{j,\nu}^{(i\mu)}{X_{i\nu}G_{\nu j}^{\left( {i\mu } \right)} X_{j\mu}}\right|   \right) \le C\varphi^{2\xi}\left[q+  \frac{1}{N}\left( {\sum^{(i\mu)}_{j,\nu } {\left| {G_{\nu j}^{(i\mu)} } \right|^2 } } \right)^{1/2} \right]  \nonumber\\
& \le C\varphi^{2\xi}\left[q+  \frac{1}{N}\left( \sum_{\nu}\left({G}^{(i\mu)}_{\nu \nu}  + \frac{\bar z}{\eta} \Im \, G_{\nu\nu}^{(i\mu)}\right)\right)^{1/2} \right]   \le C\varphi^{2\xi}\left[q+\sqrt{ \frac{|m_2^{(i\mu)}|}{N} } + \sqrt{\frac{ \Im \, m_2^{(i\mu)}}{N\eta}} \right]  ,\label{off_larged}
\end{align}
where in the third step we used (\ref{eq_gsq3}). As in (\ref{m2psi}), we can show that
\begin{equation}\label{estimatel1} \sqrt{\frac{ \Im \, m_2^{(i\mu)}}{N\eta}}= O(\Psi_\theta).\end{equation}
For the other term, we have
\begin{align}
 \sqrt{ \frac{|m_2^{(i\mu)}|}{N} } \le  \sqrt {\frac{|m_{2c}|  + |m_2^{(i\mu)}  - m_2| + |m_2  - m_{2c} |}{{N}}}   & \le C\left( \frac{1}{N\sqrt{\eta}}+\sqrt { \frac{\theta }{{N}}}  + \sqrt {\frac{{\left| {m_{2c} } \right|}}{N}}  \right)  \le C \Psi _\theta , \label{estimatel2}
\end{align}
where we used (\ref{m_T}), and that 
$$\frac{{\left| {m_{2c} } \right|}}{{N}} =O\left(\frac{{\Im\, m_{2c} }}{N\eta}\right), $$
since $|m_{2c}|= O(1)$ and $\Im\, m_{2c} \ge c\eta$ by Lemma \ref{lem_mbehavior}. Hence from (\ref{off_larged}), (\ref{estimatel1}) and (\ref{estimatel2}), we obtain that ${\mathbf 1}(\Xi) |G_{i\mu}|  \le C\varphi^{2\xi}\left(q+\Psi_\theta\right)$ with $\xi$-high probability. Together with (\ref{2blocks}), we get the estimate in (\ref{Zestimate1}) for $\Lambda_o$. Finally, the estimate (\ref{Zestimate2}) can be proved in a similar way with the bound $\mathbf{1}(\eta\ge 1)\max_{a,b}|G_{ab}|= O(1)$.
\end{proof}

Our proof of the local law starts with an analysis of the self-consistent equation. Recall that $m_{2c}(z)$ is the solution to the equation $z=f(m)$ for $f$ defined in (\ref{deformed_MP2}).

\begin{lem}\label{lemm_selfcons_weak}
Let $c_0>0$ be sufficiently small. Fix $C_0>0$, $\xi \ge 3$ and $C_1 \ge 8\xi $. Then there exists $C>0$ such that the following estimates hold uniformly in $z \in S(c_0, C_0,C_1)$ with $\xi$-high probability:
\begin{equation}
{\mathbf 1}(\eta \ge 1)\left|z- f(m_2) \right|\le C\varphi^{2\xi }(q+N^{-1/2}), \label{selfcons_lemm2}
\end{equation}
and
\begin{equation}
{\mathbf 1}(\Xi)\left|z - f(m_2) \right|\le C\varphi^{2\xi }(q+\Psi_\theta) ,  \label{selfcons_lemm}
\end{equation}
where $\Xi$ is defined in Lemma \ref{Z_lemma}. 
Moreover, we have the finer estimates
\begin{equation}
{\mathbf 1}(\Xi)\left( z - f(m_2) \right)= \mathbf 1(\Xi) \left([Z]_1 + [Z]_2\right) + O\left(\varphi^{4\xi }\left(q^2+\Psi^2_\theta\right) \right),\ \label{selfcons_improved}
\end{equation}
with $\xi_1$-high probability, where
\begin{equation}\label{def_Zaver}
[Z]_1:=\frac{1}{N}\sum_{i\in \mathcal I_1} \frac{\sigma_i}{(1+m_2\sigma_i)^2} Z_i, \ \ [Z]_2:=\frac{1}{N}\sum_{\mu \in \mathcal I_2} Z_\mu.
\end{equation}
\end{lem}

\begin{proof}
We first prove (\ref{selfcons_improved}), from which (\ref{selfcons_lemm}) follows due to (\ref{Zestimate1}) and (\ref{assumption3}). By (\ref{resolvent2}), (\ref{Zi}) and (\ref{Zmu}), we have
\begin{equation}\label{self_Gii}
\frac{1}{{G_{ii} }}=  - 1 - \frac{\sigma_i}{N} \sum_{\mu\in \mathcal I_2}G^{\left( i \right)}_{\mu\mu}+ Z_i =  - 1 - \sigma_i m_2 + \epsilon_i,
\end{equation}
and
\begin{equation}\label{self_Gmu}
\frac{1}{{G_{\mu\mu} }}=  - z - \frac{1}{N} \sum_{i\in \mathcal I_1}\sigma_i G^{\left( \mu\right)}_{ii}+ Z_{\mu} =  - z - \frac{1}{N} \sum_{i\in \mathcal I_1}\sigma_i G_{ii} + \epsilon_\mu,
\end{equation}
where 
$$\epsilon_i := Z_i + \sigma_i\left(m_2 - m_2^{(i)}\right) \ \ \text{and} \ \ \epsilon_\mu := Z_{\mu} + \frac{1}{N} \sum_{i\in \mathcal I_1}\sigma_i \left(G^{\left( \mu\right)}_{ii}-G_{ii}\right).$$
Using (\ref{m_T}), (\ref{m11_T}) and (\ref{Zestimate1}), we have for all $i$ and $\mu$,
\begin{equation}\label{erri}
\mathbf 1(\Xi)\left(|\epsilon_i | + |\epsilon_\mu| \right)\le C\varphi^{2\xi } (q + \Psi_\theta), 
\end{equation}
with $\xi $-high probability.
Then using (\ref{self_Gmu}), we get that for any $\mu$ and $\nu$,
\begin{equation}\label{diagonal_compare}
\mathbf 1(\Xi)(G_{\mu\mu} - G_{\nu\nu}) = \mathbf 1(\Xi)G_{\mu\mu} G_{\nu\nu} (\epsilon_\nu - \epsilon_\mu) = O\left(\varphi^{2\xi } (q + \Psi_\theta) \right),
\end{equation}
with $\xi $-high probability. This implies that 
\begin{equation}\label{average_bound}
\mathbf 1(\Xi)|G_{\mu\mu}-m_2| \le C\varphi^{2\xi } (q + \Psi_\theta), \ \ \mu \in \mathcal I_2,
\end{equation}
with $\xi $-high probability. 

Now we plug (\ref{self_Gii}) into (\ref{self_Gmu}) and take the average $N^{-1}\sum_\mu$. Note that we can write 
$$\frac{1}{G_{\mu\mu}} = \frac{1}{m_2} - \frac{1}{m_2^2}(G_{\mu\mu} - m_2) + \frac{1}{m_2^2}(G_{\mu\mu} - m_2)^2\frac{1}{G_{\mu\mu}}.$$
After taking the average, the second term on the right-hand side vanishes and the third term provides a $ O(\varphi^{4\xi } (q + \Psi_\theta)^2)$ factor by (\ref{average_bound}). On the other hand, using (\ref{resolvent8}) and (\ref{Zestimate1}) we get
\begin{equation*}
1(\Xi)\left| \frac{1}{N} \sum_{i\in \mathcal I_1}\sigma_i \left(G^{\left( \mu\right)}_{ii}-G_{ii}\right)\right| \le 1(\Xi) \frac{1}{N} \sum_{i\in \mathcal I_1}\sigma_i \left|\frac{G_{i\mu} G_{\mu i}}{G_{\mu\mu}}\right|\le  C\varphi^{4\xi } (q + \Psi_\theta)^2,
\end{equation*}
and
\begin{equation*}
\mathbf 1(\Xi)|m_2 - m_2^{(i)}| \le \mathbf 1(\Xi)\frac{1}{N}\sum_{\mu\in \mathcal I_2} \left|\frac{G_{\mu i} G_{i\mu}}{G_{ii}}\right| \le  C\varphi^{4\xi } (q + \Psi_\theta)^2,
\end{equation*}
with $\xi$-high probability. Hence the average of (\ref{self_Gmu}) gives
\begin{equation*}
\mathbf 1(\Xi) \frac{1}{m_2} =  \mathbf 1(\Xi)\left\{   \frac{1}{N} \sum_{i\in \mathcal I_1} \frac{\sigma_i}{1+\sigma_i m_2 - Z_i + O\left(\varphi^{4\xi } (q + \Psi_\theta)^2\right)} + [Z]_2 \right\} +  O\left(\varphi^{4\xi } (q + \Psi_\theta)^2 \right) ,
\end{equation*}
with $\xi$-high probability. Finally, using (\ref{assumption3}) and the definition of $\Xi$ we can expand the fractions in the sum to get that
\begin{align*}
\mathbf 1(\Xi)\left\{z + \frac{1}{m_2} - \frac{1}{N} \sum_{i\in \mathcal I_1} \frac{\sigma_i}{1+\sigma_i m_2}\right\} &=  \mathbf 1(\Xi)\left([Z]_1+ [Z]_2 \right) +  O\left(\varphi^{4\xi } (q + \Psi_\theta)^2 \right) .
\end{align*}
This concludes (\ref{selfcons_improved}).



Then we prove (\ref{selfcons_lemm2}). Using the bound $\mathbf{1}(\eta\ge 1)\max_{a,b}|G_{ab}|=O(1)$, it is easy to see that $m_2=O(1)$ and $\theta=O(1)$. Thus we have $1(\eta\ge 1) \Psi_\theta =O(N^{-1/2})$ and (\ref{erri}) gives
\begin{equation}\label{epsilonL}
1(\eta\ge 1) (|\epsilon_i|+|\epsilon_\mu|) \le C\varphi^{2\xi} (q + N^{-1/2}),
\end{equation}
with $\xi$-high probability. First, we claim that for $\eta \ge 1$, 
\begin{equation}\label{estimate_m2L}
|m_2| \ge \Im\, m_2 \ge c \ \text{ with }\xi\text{-high probability}, 
\end{equation} 
for some constant $c>0$. By the spectral decomposition (\ref{spectral1}), we have
$$\Im\, G_{ii} = \Im \sum\limits_{k = 1}^{M} \frac{z|\xi_k(i)|^2}{\lambda_k-z}= \sum\limits_{k = 1}^{M} |\xi_k(i)|^2 \Im \left( -1 + \frac{\lambda_k}{\lambda_k-z}\right)\ge 0.$$
Then by (\ref{self_Gmu}), $G_{\mu\mu}^{-1}$ is of order $O(1)$ and has an imaginary part $\le - \eta + O\left(\varphi^{2\xi} (q + N^{-1/2})\right)$. This implies that $ \Im\, G_{\mu\mu} \ge c\eta$ with $\xi$-high probability, which concludes (\ref{estimate_m2L}). Next, we claim that 
\begin{equation}\label{estimate_m23}
| 1+ \sigma_i m_2| \ge c \ \text{ with }\xi\text{-high probability}, 
\end{equation} 
for some constant $c>0$. In fact, if $\sigma_i \le 1/(2m_2)$, we trivially have $| 1+ \sigma_i m_2|\ge 1/2$. Otherwise, we have $\sigma_i \sim 1$ by (\ref{estimate_m2L}) and $|1+ \sigma_i m_2| \ge \sigma_i \Im\, m_2 \ge c$. Finally, with (\ref{epsilonL}), (\ref{estimate_m2L}) and (\ref{estimate_m23}), we can repeat the previous proof to get (\ref{selfcons_lemm2}).
\end{proof}

The following lemma gives the stability of the equation $z- f(m)=0$. Roughly, it states that if $z - f(m_{2}(z))$ is small and $m_2(\tilde z)-m_{2c}(\tilde z)$ is small for $\tilde z \ge z$, then $m_{2}(z)-m_{2c}(z)$ is small. For an arbitrary $z\in S(c_0,C_0, C_1)$, we define the discrete set
\begin{align}\label{eqn_def_L}
L(w):=\{z\}\cup \{z'\in S(C_1): \text{Re}\, z' = \text{Re}\, z, \text{Im}\, z'\in [\text{Im}\, z, 1]\cap (N^{-10}\mathbb N)\} .
\end{align}
Thus, if $\text{Im}\, z \ge 1$ then $L(z)=\{z\}$; if $\text{Im}\, z<1$ then $L(z)$ is a 1-dimensional lattice with spacing $N^{-10}$ plus the point $z$. Obviously, we have $|L(z)|\le N^{10}$. 

\begin{lem}\label{stability}
The self-consistent equation $z - f(m)=0$ is stable on $S(c_0,C_0, C_1)$ in the following sense. Suppose the $z$-dependent function $\delta$ satisfies $N^{-2} \le \delta(z) \le (\log N)^{-1}$ for $z\in S(c_0,C_0, C_1)$ and that $\delta$ is Lipschitz continuous with Lipschitz constant $\le N^2$. Suppose moreover that for each fixed $E$, the function $\eta \mapsto \delta(E+i\eta)$ is non-increasing for $\eta>0$. Suppose that $u_2: S(c_0,C_0,C_1)\to \mathbb C$ is the Stieltjes transform of a probability measure. Let $z\in S(c_0,C_0,C_1)$ and suppose that for all $z'\in L(z)$ we have 
\begin{equation}\label{Stability0}
\left| z- f(u_2)\right| \le \delta(z).
\end{equation}
Then we have
\begin{equation}
\left|u_2(z)-m_{2c}(z)\right|\le \frac{C\delta}{\sqrt{\kappa+\eta+\delta}},\label{Stability1}
\end{equation}
for some constant $C>0$ independent of $z$ and $N$, where $\kappa$ is defined in (\ref{KAPPA}). 
\end{lem}
\begin{proof}
This result is proved in \cite[Appendix A.2]{KY2}
\end{proof}

Note that by Lemma \ref{stability} and (\ref{selfcons_lemm2}), we immediately get that
\begin{equation}\label{average_L}
\mathbf 1(\eta\ge 1)\theta(z) \le C\varphi^{2\xi}(q+ N^{-1/2}),
\end{equation}
with $\xi$-high probability. From (\ref{Zestimate2}), we obtain the off-diagonal estimate
\begin{equation}\label{offD_L}
\mathbf 1(\eta\ge 1)\Lambda_o(z) \le C\varphi^{2\xi}(q+ N^{-1/2})
\end{equation}
with $\xi$-high probability. Using (\ref{average_bound}), (\ref{self_Gii}) and (\ref{average_L}), we get that 
\begin{equation}\label{diag_L}
\mathbf 1(\eta\ge 1)\left(\left|G_{ii}+(1+\sigma_i m_{2c})^{-1}\right| + |G_{\mu\mu}-m_{2c}|\right) \le C\varphi^{2\xi} (q + N^{-1/2}),
\end{equation}
with $\xi$-high probability, which gives the diagonal estimate. These bounds can be easily generalized to the case $\eta \ge c$ for some constant $c>0$. Comparing with (\ref{DIAGONAL}), one can see that the bounds (\ref{offD_L}) and (\ref{diag_L}) are optimal for $\eta\ge c$ case. Now it remains to deal with the small $\eta$ case (in particular, the local case with $\eta\ll 1$). We first prove the following weak bound.

\begin{lem}\label{alem_weak} 
Let $c_0>0$ be sufficiently small. Fix $C_0>0$, $\xi \ge 3$ and $C_1 \ge 8\xi$. Then there exists $C>0$ such that with $\xi_1$-high probability,
\begin{equation} \label{localweakm}
\Lambda(z) \le C\varphi^{2\xi}\left(\sqrt{q}+ (N\eta)^{-1/3}\right),
\end{equation}
holds uniformly in $z \in S(c_0,C_0,C_1)$.
\end{lem}
\begin{proof}
One can prove this lemma using a continuity argument as in \cite[Section 4.1]{BEKYY} or \cite[Section 3.6]{EKYY1}. The key inputs are Lemmas \ref{Z_lemma}-\ref{stability}, diagonal estimate (\ref{average_bound}) and the estimates (\ref{average_L})-(\ref{diag_L}) for the $\eta \ge 1$ case. All the other parts of the proof are essentially the same. 
\end{proof}

To get the strong local laws given in Lemma \ref{lem_small}, we need stronger bounds on $[Z]_1$ and $[Z]_2$ in (\ref{selfcons_improved}). They follow from the abstract decoupling lemma (or the ``fluctuation averaging lemma") below. 

\begin{lem} \label{abstractdecoupling}
Fix a constat $\xi>0$. Suppose $q\le \varphi^{-5\xi}$ and that there exists $\tilde S\subseteq S(c_0,C_0,L)$ with $L\ge 18\xi$ such that we have with $\xi$-high probability
\begin{equation} 
\Lambda(z) \le \gamma(z) \text{ for } z\in \tilde S,
\end{equation}
where $\gamma$ is a deterministic function satisfying $\gamma(z)\le \varphi^{-\xi}$. Then we have with $(\xi-\tau_N)$-high probability,
\begin{equation}
\left|[Z]_1(z)\right|+ \left|[Z]_2(z)\right| \le \varphi^{18\xi} \left(q^2 + \frac{1}{(N\eta)^2} + \frac{\Im \, m_{2c}(z) + \gamma(z)}{N\eta} \right),
\end{equation}
for $z\in \tilde D$, where $\tau_N:=2/\log \log N$. 
\end{lem}
\begin{proof}
The bound for $[Z]_2$ is proved in Lemma 4.1 of \cite{EKYY1}. The bound for $[Z]_1$ can be proved in a similar way, except that the coefficients ${\sigma_i}/{(1+m_2\sigma_i)^2}$ are random and depend on $i$. This can be dealt with by writing, for any $i\in \mathcal I_1$,
$$m_2 = m_2^{(i)} + \frac{1}{N}\sum_{\mu\in\mathcal I_2} \frac{G_{\mu i} G_{i\mu}}{G_{ii}} = m_2^{(i)} + O(\Lambda_o^2).$$
Then on the event $\Xi$, we have with $\xi$-high probability,
\begin{align}
[Z]_1 & =\frac{1}{N}\sum_{i\in \mathcal I_1} \frac{\sigma_i}{\left(1+m_2^{(i)}\sigma_i \right)^2} Z_i + O(\Lambda_o^2) = \frac{1}{N}\sum_{i\in \mathcal I_1} (1-\mathbb E_i)\left[\frac{\sigma_i}{\left(1+m_2^{(i)}\sigma_i\right)^2}G_{ii}^{-1}\right]+ O(\Lambda_o^2) \nonumber\\
& = \frac{1}{N}\sum_{i\in \mathcal I_1} (1-\mathbb E_i)\left[\frac{\sigma_i}{\left(1+m_2 \sigma_i\right)^2}G_{ii}^{-1}\right]+ O\left(\varphi^{4\xi}\left(q^2 + \frac{1}{(N\eta)^2} + \frac{\Im \, m_{2c}(z) + \gamma(z)}{N\eta}\right)\right), \label{Z1_aver}
\end{align}
where in the last step we used (\ref{Zestimate1}). Then the proof for the first term in (\ref{Z1_aver}) is a slight modification of the one in \cite{EKYY1} or the simplified proof given in \cite[Appendix B]{EKYYL}. Finally, we can use that the event $\Xi$ holds with $\xi$-high probability by Lemma \ref{alem_weak}. For a demonstration of the above process, one can also refer to the proof of Lemma 4.9 of \cite{XYY}.
\end{proof}


\begin{proof}[Proof of the local deformed MP laws (\ref{MPBOUNDS}) and (\ref{DIAGONAL})]
Fix $c_0,C_0>0$, $\xi> 3$ and set 
$$L:=120\xi, \ \ \tilde \xi:= 2/\log 2 + \xi.$$
Hence we have $\tilde \xi \le 2\xi$ and $L\ge 60\tilde \xi$. Then to prove (\ref{DIAGONAL}), it suffices to prove 
\begin{equation}\label{goal_law1}
\bigcap_{z \in S(c_0,C_0,L)} \left\{ \Lambda(z) \leq C\varphi^{20\tilde \xi}\left(q+ \sqrt{\frac{\operatorname{Im} m_{2c}(z) }{N \eta}}+ \frac{1}{N\eta}\right) \right\},
\end{equation}
with $\xi$-high probability. 

By Lemmas \ref{alem_weak}, we have that $\Xi$ holds with $\tilde \xi$-high probaility. Then together with Lemma \ref{abstractdecoupling} and (\ref{selfcons_improved}), we get that with $(\tilde \xi - \tau_N)$-high probability, 
\begin{align*}
\left| z - f(m_2) \right| & \le \varphi^{18\tilde \xi} \left[q^2 + \frac{1}{(N\eta)^2} + \frac{\Im \, m_{2c}+ C\varphi^{2\tilde \xi}(\sqrt{q}+ (N\eta)^{-1/3})}{N\eta} \right] \\
& \le C\left[\varphi^{20\tilde \xi } \left(q^2 + \frac{1}{(N\eta)^{4/3}}\right) + \varphi^{18\tilde \xi }  \frac{\Im \, m_{2c}}{N\eta} \right] ,
\end{align*}
where we used Young's inequality for the $\sqrt{q}/(N\eta)$ term. Now applying Lemma \ref{stability}, we get that with $(\tilde \xi - \tau_N)$-high probability, 
\begin{equation}
\theta = |m_2 - m_{2c}| \le C\varphi^{10\tilde \xi} \left(q + \frac{1}{(N\eta)^{2/3}}\right) +C\varphi^{18\tilde \xi} \frac{\Im \, m_c}{N\eta\sqrt{\kappa+ \eta}} \le C\varphi^{18\tilde \xi} \left(q + \frac{1}{(N\eta)^{2/3}}\right),
\end{equation}
where we used (\ref{SQUAREROOTBEHAVIOR}) in the second step. 
Then using Lemma \ref{Z_lemma}, (\ref{self_Gii}) and (\ref{average_bound}), it is easy to obtain that
\begin{align*}
\Lambda \le C\varphi^{2\tilde \xi}(q+\Psi_\theta) + \theta & \le C\varphi^{18\tilde \xi} \left(q + \frac{1}{(N\eta)^{2/3}}\right) + C\varphi^{2\tilde \xi} \sqrt{\frac{\Im \, m_{2c}}{N\eta}}  \le \varphi^{20\tilde \xi} \left(q + \frac{1}{(N\eta)^{2/3}}\right)+ \varphi^{3\tilde \xi}\sqrt{\frac{\Im \, m_{2c}}{N\eta}},
\end{align*}
uniformly in $z\in S(c_0,C_0,L)$ with $(\tilde \xi-\tau_N)$-high probability, which is a better bound than the one in (\ref{localweakm}). We can repeat this process $M$ times, each iteration yields a stronger bound on $\Lambda$ which holds with a  smaller probability. More specifically, suppose that after $k$ iterations we get the bound
\begin{align}
\Lambda \le \varphi^{20\tilde \xi} \left(q + \frac{1}{(N\eta)^{1-\tau}}\right)+ \varphi^{3\tilde \xi}\sqrt{\frac{\Im \, m_{2c}}{N\eta}} ,\label{iteration1}
\end{align}
uniformly in $z\in S(c_0,C_0,L)$ with $\tilde \xi'$-high probability. Then by Lemma \ref{abstractdecoupling} and (\ref{selfcons_improved}), we have with $(\tilde \xi' - \tau_N)$-high probability, 
\begin{align*}
\left| z -f(m_2) \right| & \le \varphi^{18\tilde \xi} \left[q^2 + \frac{1}{(N\eta)^2} + \frac{\Im \, m_{2c}}{N\eta} + \frac{\varphi^{20\tilde \xi}}{N\eta}\left(q + \frac{1}{(N\eta)^{1-\tau}}\right)+ \frac{\varphi^{3\tilde \xi}}{N\eta}\sqrt{\frac{\Im \, m_{2c}}{N\eta}} \right] \\
& \le C\left[\varphi^{38\tilde \xi } \left(q^2 + \frac{1}{(N\eta)^{2-\tau}}\right) + \varphi^{18\tilde \xi }  \frac{\Im \, m_{2c}}{N\eta} \right] .
\end{align*}
Then using Lemma \ref{stability}, we get that with $(\tilde \xi' - \tau_N)$-high probability, 
\begin{equation}
\theta \le C\varphi^{19\tilde \xi} \left(q + \frac{1}{(N\eta)^{1-\tau/2}}\right) +C\varphi^{18\tilde \xi} \frac{\Im \, m_c}{N\eta\sqrt{\kappa+ \eta}} \le C\varphi^{19\tilde \xi} \left(q + \frac{1}{(N\eta)^{1-\tau/2}}\right).
\end{equation}
Again with Lemma \ref{Z_lemma}, (\ref{self_Gii}) and (\ref{average_bound}), we obtain that
\begin{align}
\Lambda &\le C\varphi^{2\tilde \xi}(q+\Psi_\theta) + \theta  \le C\varphi^{19\tilde \xi} \left(q + \frac{1}{(N\eta)^{1-\tau/2}}\right) + C\varphi^{2\tilde \xi} \sqrt{\frac{\Im \, m_c}{N\eta}}  \nonumber\\
& \le \varphi^{20\tilde \xi} \left(q + \frac{1}{(N\eta)^{1-\tau/2}}\right)+ \varphi^{3\tilde \xi}\sqrt{\frac{\Im \, m_c}{N\eta}}, \label{iteration2}
\end{align}
uniformly in $z\in S(c_0,C_0,L)$ with $(\tilde \xi' -\tau_N)$-high probability. Comparing with (\ref{iteration1}), we see that the power of $(N\eta)^{-1}$ is increased from $1-\tau$ to $1-\tau/2$, and moreover, there is no extra constant $C$ appearing on the right-hand side of (\ref{iteration2}). Thus after $M$ iterations, we get
\begin{align}
\Lambda &\le \varphi^{20\tilde \xi} \left(q + \frac{1}{(N\eta)^{1-(1/2)^{M-1}/3}}\right)+ \varphi^{3\tilde \xi}\sqrt{\frac{\Im \, m_c}{N\eta}}, 
\end{align}
uniformly in $z\in S(c_0,C_0,L)$ with $(\tilde \xi' -M\tau_N)$-high probability. Taking $M=\left\lfloor \log\log N / \log 2\right\rfloor$ such that
$$\tilde \xi - M\tau_N \ge \xi, \ \ \frac{1}{(N\eta)^{-(1/2)^{M-1}/3}} \le (N\eta)^{4/(3\log N)} \le C,$$
we can conclude (\ref{goal_law1}) and hence (\ref{DIAGONAL}). Finally to prove (\ref{MPBOUNDS}), we only need to plug (\ref{goal_law1}) into Lemma \ref{abstractdecoupling} and then apply Lemma \ref{stability}.
\end{proof}

Finally, we describe briefly the proof of (\ref{boundH})-(\ref{SEC}).

\begin{proof}[Proof of (\ref{boundH})]
The norm bound on $H$ follows from a standard application of the moment method, for example, see \cite[Lemma 4.3]{EKYY1} or \cite{Handbook_DS}.
\end{proof}

\begin{proof}[Proof of (\ref{delocal})] Choose $z_0=E+i\eta_0\in S(2c_1, C_0, C_1)$ with $\eta_0 = \varphi^{C_1} N^{-1}$. By (\ref{DIAGONAL}), we have
\begin{equation*}
\vert G_{\mu \mu}(z_0)\vert= O(1) \text{ with } \xi_1\text{-high probability.}
\end{equation*}
Then using the spectral decomposition (\ref{spectral1}), we get
\begin{equation}\label{spectraldecomp}
\sum_{k=1}^N \frac{\eta_0 \vert \zeta_k(\mu) \vert^2}{(\lambda_k-E)^2+\eta_0^2}  = \operatorname{Im} {G}_{\mu \mu}(z_0) =  O(1).
\end{equation}
By (\ref{boundH}), $\lambda_k + i\eta_0 \in S(2c_1, C_0, C_1)$ with $\xi_1$-high probability if we choose $C_0$ large enough. Then choosing $E=\lambda_k$ in (\ref{spectraldecomp}) yields that, with $\xi_1$-high probability,
\begin{equation*}
\vert \zeta_k(\mu) \vert^2 \le \eta_0 = \frac{\varphi^{C_1}}{N},
\end{equation*} 
for all $k$. The proof for $|\xi_k(i)|^2$ is similar.
\end{proof}

\begin{proof} [Proof of (\ref{SEC})]
With (\ref{MPBOUNDS}) and (\ref{DIAGONAL}), (\ref{SEC}) follows from a routine application of the Helffer-Sj{\"o}strand functional calculus; we refer the reader to \cite{EYY}.
\end{proof}

\end{appendix}

\vspace{5pt}

\noindent{\bf Acknowledgements.}  The authors would like to thank Jeremy Quastel and Jun Yin for fruitful discussions and valuable suggestions, which have significantly improved the paper. The first author also want to thank Jun Yin for the hospitality when he visited Madison.


\end{document}